\documentclass[parskip=never,abstract=true]{scrartcl}

\usepackage[utf8]{inputenc}

\usepackage{csquotes}

\usepackage[T1]{fontenc}

\usepackage{bbm}

\usepackage{lmodern}
\usepackage{yfonts}
\usepackage{textcomp}
\usepackage{setspace}
\usepackage[a4paper,left=2.5cm,right=2cm,top=2.5cm,bottom=2.5cm]{geometry}

\usepackage{extarrows}
\usepackage{etoolbox}
\usepackage{xparse}

\usepackage{mathtools}
\usepackage{amsthm,thmtools}
\usepackage{dsfont}
\let\mathbb\mathds
\usepackage{amssymb}
\usepackage[scr]{rsfso} % rather than mathrsfs
\usepackage{bm}
\usepackage{euscript}
\usepackage{amsbsy}
\DeclareMathAlphabet\mathbfcal{OMS}{cmsy}{b}{n}

\usepackage{microtype}
\usepackage{authblk}
\usepackage{tikz}
\usetikzlibrary{cd,quotes}
\usetikzlibrary{decorations.markings}
\usetikzlibrary{decorations.pathreplacing,calligraphy}
\usepackage{pgfplots}
\pgfplotsset{compat=1.13}
\usepackage{booktabs}
\usepackage{float}
\usepackage{environ}
\usepackage{xcolor,soul}
\usepackage{hyperref} % éste va siempre el último (por lo general)
\hypersetup{%
  colorlinks,%
  linkcolor={red!60!black},%
  citecolor={blue!50!black},%
  urlcolor={blue!80!black}%
}
\usepackage{enumitem,cleveref}
\usepackage{eso-pic}

\usepackage[normalem]{ulem}

\usetikzlibrary{calc}
\usetikzlibrary{decorations.pathmorphing}

\tikzset{curve/.style={settings={#1},to path={(\tikztostart)
    .. controls ($(\tikztostart)!\pv{pos}!(\tikztotarget)!\pv{height}!270:(\tikztotarget)$)
    and ($(\tikztostart)!1-\pv{pos}!(\tikztotarget)!\pv{height}!270:(\tikztotarget)$)
    .. (\tikztotarget)\tikztonodes}},
    settings/.code={\tikzset{quiver/.cd,#1}
        \def\pv##1{\pgfkeysvalueof{/tikz/quiver/##1}}},
    quiver/.cd,pos/.initial=0.35,height/.initial=0}

\tikzset{tail reversed/.code={\pgfsetarrowsstart{tikzcd to}}}
\tikzset{2tail/.code={\pgfsetarrowsstart{Implies[reversed]}}}
\tikzset{2tail reversed/.code={\pgfsetarrowsstart{Implies}}}
\tikzset{no body/.style={/tikz/dash pattern=on 0 off 1mm}}

\usepackage{tikz-cd}
\tikzset{Rightarrow/.style={double equal sign distance,>={Implies},->},
  triple/.style={-,preaction={draw,Rightarrow}},
  quadruple/.style={preaction={draw,Rightarrow,shorten >=0pt},shorten >=1pt,-,double,double
    distance=0.2pt}}

\def\on{\operatorname}

\def\CC{\mathbb{C}}

\def\C{\EuScript{C}}
\def\D{\EuScript{D}}

\def\DD{\mathbb{D}}
\def\bS{\textbf{S}}
\def\FF{\mathbb{F}}

\def\Fun{\on{Fun}}
\def\Cat{\on{Cat}}

\def\Hom{\on{Hom}}
\def\D{\EuScript{D}}
\def\Nerv{\on{N}}
\def\Nsc{\on{N}^{\on{sc}}}

\def\Set{\on{Set}}

\def\scsSet{{\on{Set}_{\Delta}^{\mathbf{sc}}}}
\def\mbsSet{{\Set_\Delta^{\mathbf{mb}}}}

\def\ST{\mathbb{S}\!\on{t}}

\def\UN{\mathbb{U}\!\on{n}}
\def\sc{{\on{sc}}}
\SetEnumitemKey{axioms}{itemsep=0pt,align=left,itemindent=!,
                        listparindent=\parindent,parsep=\parskip,
                        before=,
                        label={\Roman*}}

\SetEnumitemKey{claims}{itemsep=0pt,align=left,itemindent=!,
                        listparindent=\parindent,parsep=\parskip,
                        before=,
                        label={\Roman*}}

\DeclarePairedDelimiterX\set[1]{\lbrace}{\rbrace}{#1}
\newlist{implications}{description}{1} % Mejorar todo este sistema
\setlist[implications]{itemsep=0pt,leftmargin=\parindent}
\NewDocumentCommand\implication{o}
  {\IfValueTF{#1}
    {\auximplication#1\relax}
    {\item[\normalfont($\,\Rightarrow\,$)]}}
\NewDocumentCommand\auximplication{u-u\relax}
  {\item[\normalfont(#1)$\,\Rightarrow\,$(#2)]}

\makeatletter
\newcounter{diagram}[section]
\def\thediagram{\thesection.\arabic{diagram}}
\def\ftype@diagram{4}
\def\ext@diagram{diag}
\def\fnum@diagram{Diagram~\thediagram}
\def\fs@diagram{htbp!}
\ExplSyntaxOn
\NewDocumentEnvironment{diagram}{O{htbp!}m}
  {\@float{diagram}[#1]\centering}
  {
   
   \caption{}
   \label{#2}
   \end@float
  }

\newcounter{subdiagram}[diagram]
\def\thesubdiagram{\thediagram.\arabic{subdiagram}}
\NewEnviron{multidiagram}{\expandafter\domultidiagram\BODY\enddomultidiagram}
\NewDocumentCommand\domultidiagram{omu\enddomultidiagram}
 {
  \IfValueTF{#1}{\diagram[#1]}{\diagram}{}
   \refstepcounter{diagram}
   \centering
    \seq_clear:N \l_tmpb_seq
    \seq_set_split:Nnn \l_tmpa_seq { \next } { #3 }
    \seq_map_inline:Nn \l_tmpa_seq
     {
      \seq_put_right:Nn \l_tmpb_seq
       {
        \begin{tabular}[b]{@{}c@{}}
         ##1 \\[3ex]
         \refstepcounter{subdiagram}
         \label{#2\othercolon\the\value{subdiagram}}
         Diagram~\thesubdiagram 
        \end{tabular}
       }
     }
    \seq_use:Nn \l_tmpb_seq { \qquad }
   \let\label\@gobble
   \let\caption\@gobble
  \enddiagram
 }
\ExplSyntaxOff
\def\othercolon{:}
\makeatother

\declaretheoremstyle[bodyfont=\itshape,notefont=\bfseries]{abellanA}
\declaretheoremstyle[notefont=\bfseries]{abellanB}

\declaretheorem[style=abellanA,numberwithin=section,name={Theorem}]{theorem}

\declaretheorem[style=abellanA,numberlike=theorem,name={Lemma}]{lemma}
\declaretheorem[style=abellanB,numberlike=theorem,name={Definition}]{definition}
\declaretheorem[style=abellanB,numberlike=theorem,name={Remark}]{remark}
\declaretheorem[style=abellanB,numberlike=theorem,name={Construction}]{construction}
\declaretheorem[style=abellanA,numberlike=theorem,name={Proposition}]{proposition}

\declaretheorem[style=abellanB,numbered=no,name={Notation}]{notation}
\declaretheorem[style=abellanA,numberlike=theorem,name={Corollary}]{corollary}

\newtheorem*{thm*}{Theorem}
\newtheorem*{prop*}{Proposition}
\newtheorem*{cor*}{Corollary}

\docsvlist{theorem,lemma,definition,remark,proposition,example,corollary}

\let\leq\leqslant
\let\geq\geqslant
\let\epsilon\varepsilon
\let\isom\simeq

\newcommand*\mathblank{\mathord{-}}

\DeclareMathOperator*\colimdag{colim^\dagger}

\DeclareMathOperator*\colim{colim}

\def\msSet{{\on{Set}_{\Delta}^+}}
\def\mssSet{{\on{Set}^{\mathbf{ms}}_{\Delta}}}

\def\Cat{\on{Cat}}

\usepackage[bbgreekl]{mathbbol}

\let\emptyset\varnothing

\makeatletter
\newcommand{\fixed@sra}{$\vrule height 2\fontdimen22\textfont2 width 0pt\rightarrow$}
\newcommand{\shortarrowup}[1]{%
  \mathrel{\text{\rotatebox[origin=c]{65}{\fixed@sra}}}
}
\newcommand{\shortarrowdown}[1]{%
  \mathrel{\text{\rotatebox[origin=c]{250}{\fixed@sra}}}
}
\makeatother

\newcommand{\upslash}{\!\shortarrowup{1}}

\def\lra{\longrightarrow}
\def\lla{\longleftarrow}

\def\llra{\def\arraystretch{.1}\begin{array}{c} \lra \\ \lla \end{array}}

\tikzcdset{
  arrows={},
  nat/.style={Rightarrow},
  nat1/.style={nat,shift left=.8ex},
  nat2/.style={nat,shift right=.8ex,swap},
  map/.style={},
  map1/.style={map,shift left=.7ex},
  map2/.style={map,shift right=.7ex,swap},
  mapA/.style={map,shift left=.5ex},
  mapB/.style={map,shift left=.5ex},
  unique/.style={line cap=round,densely dotted},
  incl/.style={hookrightarrow},
  mono/.style={rightarrowtail},
  epi/.style={twoheadrightarrow},
  iso/.style={map,sloped,"{\tikz[x=.6ex,y=.3ex]{\draw[-](0,0)sin(1,1)cos(2,0)sin(3,-1)cos(4,0);}}"},
  line/.style={dash},
  line1/.style={line,shift left=.3ex},
  line2/.style={line,shift right=.3ex,swap},
}

\def\op{{\on{op}}}

\makeatletter
\newcommand*\dirlim{\mathop{\mathpalette\varlim@{\rightarrowfill@\scriptscriptstyle}}\nmlimits@}
\newcommand*\prolim{\mathop{\mathpalette\varlim@{\leftarrowfill@\scriptscriptstyle}}\nmlimits@}
\makeatother

\def\llra{\def\arraystretch{.1}\begin{array}{c} \lra \\ \lla \end{array}}

\newcommand{\nat}{\Rightarrow}

\tikzset{
  abellanarrows/.style={line cap=round,line join=round,line width=.4pt},
  abellanarrowlength/.store in=\abellanarrowlength,
}

\ExplSyntaxOn

\cs_generate_variant:Nn \tl_replace_all:Nnn { Nf }

\NewDocumentCommand \func { s O{} m }
 {
  \group_begin:
   \IfBooleanTF{#1}
    { \keys_set:nn { abellan / func } { aligned = true , #2 } }
    { \keys_set:nn { abellan / func } {#2} }
   \abellan_func:n {#3}
  \group_end:
 }
\NewDocumentCommand \arr { s o m }
 {
  \IfBooleanF{#1}
   { \bool_if:NT \l_abellan_aligned_bool { & } }
  \abellan_arr:n {#3}
 }
\NewDocumentCommand \addarr { o m m }
 {
  \keys_set:nn { abellan / func / addarrow } { name = {#2} , #3 }
  \tl_clear:N \l_abellan_arrname_tl
 }
\NewDocumentCommand \setupfunc { m } { \keys_set:nn { abellan / func } {#1} }

\bool_new:N \l_abellan_aligned_bool
\tl_new:N \l_abellan_func_tl
\tl_new:N \l_abellan_arrname_tl
\tl_new:N \l_abellan_arrmode_tl
\tl_new:N \g_abellan_func_substitutions_tl
\quark_new:N \q_abellan

\keys_define:nn { abellan / func }
 {
  aligned .bool_set:N = \l_abellan_aligned_bool ,
  aligned .initial:n = false ,
  mode .choices:nn = { base , tikz } { \tl_set_eq:NN \l_abellan_arrmode_tl \l_keys_choice_tl } ,
  mode .initial:n = base ,
  tikzarrowlength .code:n = \tikzset{abellanarrowlength={#1}} ,
  tikzarrowlength .initial:n = 2em ,
  tikzarrowstyle .code:n = \tikzset{abellanarrows/.append ~ style={#1}} ,
 }
\keys_define:nn { abellan / func / addarrow }
 {
  name .tl_set:N = \l_abellan_arrname_tl ,
  base .code:n =
    \abellan_addarrow:nnww { \l_abellan_arrname_tl } { base } #1 \q_abellan ,
  tikz .code:n =
    \abellan_addarrow:nnww { \l_abellan_arrname_tl } { tikz } #1 \q_abellan ,
  substitutions .code:n =
   \tl_map_inline:nn {#1}
    {
     \tl_gput_right:Nx \g_abellan_func_substitutions_tl
      {
       \exp_not:n { \tl_replace_all:Nnn \l_abellan_func_tl {##1} }
        { \arr{\l_abellan_arrname_tl} }
      }
    } ,
 }
\cs_new_protected:Npn \abellan_func:n #1
 {
  \resetdynamicto
  \tl_set:Nn \l_abellan_func_tl {#1}
  \tl_replace_all:Nfn \l_abellan_func_tl
   { \char_generate:nn { `: } { 12 } } { \colon }
  \bool_if:NTF \l_abellan_aligned_bool
   { \tl_replace_all:Nnn \l_abellan_func_tl { ; } { \\ } }
   { \tl_replace_all:Nnn \l_abellan_func_tl { ; } { ,\ } }
  \tl_use:N \g_abellan_func_substitutions_tl
  \bool_if:NT \l_abellan_aligned_bool { \! \begin {aligned} \scan_stop: }
  \tl_use:N \l_abellan_func_tl
  \bool_if:NT \l_abellan_aligned_bool { \\ \end {aligned} }
 }
\NewDocumentCommand \abellan_addarrow:nnww { m m O{} u\q_abellan }
 {
  \exp_args:Nc \NewDocumentCommand { abellan_arr_#1_#2:w } { #3 }
   {
    \use:c { abellan_arr_ \l_abellan_arrmode_tl :n } { #4 }
   }
 }
\cs_new_protected:Npn \abellan_arr:n #1
 {
  \use:c { abellan_arr_#1_\l_abellan_arrmode_tl :w }
 }
\cs_new_protected:Npn \abellan_arr_base:n #1
 {
  \mathrel{#1}
 }
\cs_new_protected:Npn \abellan_arr_tikz:n #1
 {
  \mathrel{\vcenter{\hbox{\tikz[abellanarrows]{#1}}}}
 }
\ExplSyntaxOff

\addarr{monomorphism}
 {
  substitutions = {>->}{ mono }{\rightarrowtail}{\mono} ,
  base = [o]{\IfValueTF{#1}{\rightarrowtail{#1}}{\mono}} ,
  tikz = [o]{\draw[>->,inner sep=0pt]
     (0,0)--({max(\abellanarrowlength,.5em\IfValueT{#1}{+width("$\scriptstyle#1$")})},0)
     \IfValueT{#1}{node[midway,above=.7ex]{\smash{$\scriptstyle#1$}}
     node[midway,below=.7ex]{}}
     ;} ,
  }
\addarr{mapsto}
 {
  substitutions = {\mapsto}{ mapsto } ,
  base = \mapsto ,
  tikz = {\draw[|->](0,0)--(\abellanarrowlength,0);} ,
 }
\addarr{rrarrows} 
 {
  substitutions = {\rightrightarrows}{ doublerr }{\doublerr} ,
  base = \rightrightarrows ,
  tikz = {\draw[>={To[width=.8ex,length=.4ex]},->](0,0)--(\abellanarrowlength,0);
          \draw[yshift=.8ex,>={To[width=.8ex,length=.4ex]},->](0,0)--(\abellanarrowlength,0);} ,
 }
\addarr{inclusion}
 {
  substitutions = {\hookrightarrow}{\incl}{ incl } ,
  base = [o]{\IfValueTF{#1}{\hookrightarrow{#1}}{\incl}} ,
  tikz = {\draw[{Hooks[right]}->](0,0)--(\abellanarrowlength,0);} ,
 }
\addarr{natural}
 {
 substitutions = {\Rightarrow}{=>}{ nat }{\nat} ,
  base = [o]{\IfValueTF{#1}{\xRightarrow{#1}}{\nat}} ,
  tikz = [o]{\draw[line cap=butt,double equal sign distance,-Implies,inner sep=0pt]
    (0,0)--({max(\abellanarrowlength,.3em\IfValueT{#1}{+width("$\scriptstyle#1$")})},0)
    \IfValueT{#1}{node[midway,above=.7ex]{\smash{$\scriptstyle#1$}}
    node[midway,below=.7ex]{}}
    ;} ,
 }
\addarr{epimorphism} 
 {
  substitutions = {\twoheadrightarrow}{->>}{ epi }{\epi},
   base = [o]{\IfValueTF{#1}{\xtwoheadrightarrow{#1}}{\epi}} ,
   tikz = [o]{\draw[->>,inner sep=0pt]
      (0,0)--({max(\abellanarrowlength,.5em\IfValueT{#1}{+width("$\scriptstyle#1$")})},0)
      \IfValueT{#1}{node[midway,above=.7ex]{\smash{$\scriptstyle#1$}}
      node[midway,below=.7ex]{}}
      ;} ,
   }
\addarr{to}
 {
  substitutions = {\to}{\rightarrow}{ to }{ funct }{ map }{->} ,
  base = [o]{\IfValueTF{#1}{\xrightarrow{#1}}{\to}} ,
  tikz = [o]{\draw[->,inner sep=0pt]
    (0,0)--({max(\abellanarrowlength,.5em\IfValueT{#1}{+width("$\scriptstyle#1$")})},0)
    \IfValueT{#1}{node[midway,above=.7ex]{\smash{$\scriptstyle#1$}}
    node[midway,below=.7ex]{}}
    ;} ,
 }
\newcommand*\resetdynamicto
  {\gdef\dynamicto{\arr*{to}\gdef\dynamicto{\arr*{mapsto}}}}
\resetdynamicto

\setupfunc{mode=tikz}

\ExplSyntaxOn
\NewDocumentCommand \printheader { m o m }
 {
  \par\noindent
  \begin{minipage}[t]{\textwidth}\noindent
  
  \begin{tabular}[t]{ll}
     & \keyval_parse:NNn \abellan_printname:n \abellan_printnamemail:nn { #3 } %\\[-2ex]
  \end{tabular}
  \vspace{.4cm}
  \end{minipage}
  \begin{center}\Large\bfseries
   #1 \IfValueT{#2}{\\[1ex] \large #2}
  \end{center}
  \vspace{.6cm}
 }
\tl_new:N \g_abellan_autores_tl
\tl_new:N \g_abellan_asignatura_tl
\cs_new_protected:Npn \abellan_printnamemail:nn #1 #2
 {
  #1 \\ & \quad\texttt{#2} \\ &
 }
\cs_new_protected:Npn \abellan_printname:n #1
 {
  #1 \\ &
 }
\ExplSyntaxOff

\makeatletter
\tikzcdset{
  open/.code     = {\tikzcdset{hook, circled};},
  open'/.code    = {\tikzcdset{hook', circled};},
  circled/.code  = {\tikzcdset{markwith = {\draw (0,0) circle (.375ex);}};},
  markwith/.code ={
    \pgfutil@ifundefined%
    {tikz@library@decorations.markings@loaded}%
    {\pgfutil@packageerror{tikz-cd}{You need to say %
        \string\usetikzlibrary{decorations.markings} to use arrows with markings}{}}{}%
    \pgfkeysalso{/tikz/postaction = {
        /tikz/decorate,
        /tikz/decoration={markings, mark = at position 0.5 with {#1}}}
    }
  },
}
\makeatother

\def\mbsSet{\on{Set}_{\Delta}^{\mathbf{mb}}}
\def\bS{\textbf{MB}}

\def\scr{\EuScript}

\makeatletter
\newcommand{\myitem}[1]{%
  \item[#1]\protected@edef\@currentlabel{#1}%
}
\makeatother

\makeatletter
\DeclareSymbolFont{lettersA}{U}{txmia}{m}{it}

\DeclareRobustCommand*{\varmathbb}[1]{\gdef\F@ntPrefix{m@thbbch@r}%
	\@EachCharacter #1\@EndEachCharacter}

\long\def\DoLongFutureLet #1#2#3#4{% 
	\def\@FutureLetDecide{#1#2\@FutureLetToken
		\def\@FutureLetNext{#3}\else
		\def\@FutureLetNext{#4}\fi\@FutureLetNext}
	\futurelet\@FutureLetToken\@FutureLetDecide}
\def\DoFutureLet #1#2#3#4{\DoLongFutureLet{#1}{#2}{#3}{#4}}
\def\@EachCharacter{\DoFutureLet{\ifx}{\@EndEachCharacter}%
	{\@EachCharacterDone}{\@PickUpTheCharacter}}
\def\m@keCharacter#1{\csname\F@ntPrefix#1\endcsname}
\def\@PickUpTheCharacter#1{\m@keCharacter{#1}\@EachCharacter}
\def\@EachCharacterDone \@EndEachCharacter{}

\DeclareMathSymbol{\m@thbbch@rA}{\mathord}{lettersA}{129}
\DeclareMathSymbol{\m@thbbch@rB}{\mathord}{lettersA}{130}
\DeclareMathSymbol{\m@thbbch@rC}{\mathord}{lettersA}{131}
\DeclareMathSymbol{\m@thbbch@rD}{\mathord}{lettersA}{132}
\DeclareMathSymbol{\m@thbbch@rE}{\mathord}{lettersA}{133}
\DeclareMathSymbol{\m@thbbch@rF}{\mathord}{lettersA}{134}
\DeclareMathSymbol{\m@thbbch@rG}{\mathord}{lettersA}{135}
\DeclareMathSymbol{\m@thbbch@rH}{\mathord}{lettersA}{136}
\DeclareMathSymbol{\m@thbbch@rI}{\mathord}{lettersA}{137}
\DeclareMathSymbol{\m@thbbch@rJ}{\mathord}{lettersA}{138}
\DeclareMathSymbol{\m@thbbch@rK}{\mathord}{lettersA}{139}
\DeclareMathSymbol{\m@thbbch@rL}{\mathord}{lettersA}{140}
\DeclareMathSymbol{\m@thbbch@rM}{\mathord}{lettersA}{141}
\DeclareMathSymbol{\m@thbbch@rN}{\mathord}{lettersA}{142}
\DeclareMathSymbol{\m@thbbch@rO}{\mathord}{lettersA}{143}
\DeclareMathSymbol{\m@thbbch@rP}{\mathord}{lettersA}{144}
\DeclareMathSymbol{\m@thbbch@rQ}{\mathord}{lettersA}{145}
\DeclareMathSymbol{\m@thbbch@rR}{\mathord}{lettersA}{146}
\DeclareMathSymbol{\m@thbbch@rS}{\mathord}{lettersA}{147}
\DeclareMathSymbol{\m@thbbch@rT}{\mathord}{lettersA}{148}
\DeclareMathSymbol{\m@thbbch@rU}{\mathord}{lettersA}{149}
\DeclareMathSymbol{\m@thbbch@rV}{\mathord}{lettersA}{150}
\DeclareMathSymbol{\m@thbbch@rW}{\mathord}{lettersA}{151}
\DeclareMathSymbol{\m@thbbch@rX}{\mathord}{lettersA}{152}
\DeclareMathSymbol{\m@thbbch@rY}{\mathord}{lettersA}{153}
\DeclareMathSymbol{\m@thbbch@rZ}{\mathord}{lettersA}{154}

\makeatother 
\def\bcat{\varmathbb}
\def\Fr{\on{Fr}}

\title{On cofinal functors of $\infty$-bicategories}
\author{Fernando Abellán \& Walker H. Stern}
\date{} 
\begin{document}
  \maketitle
  \begin{abstract}
  	In this work, we study the notion of cofinal functor of $\infty$-bicategories with respect to the theory of partially lax colimits. The main result of this paper is a characterization of cofinal functors of $\infty$-bicategories via generalizations of the conditions of Quillen's Theorem A. As a key ingredient for the proof of our main theorem we produce for every functor of $\infty$-bicategories $f:\bcat{C} \to \bcat{D}$ an outer 2-Cartesian fibration $\mathbb{F}(\bcat{C})\to \bcat{D}$ which we identify it as the free fibration on the functor $f$. 
  	\par\vskip\baselineskip\noindent
  	\textbf{Keywords}: Partially lax colimit, cofinality, Grothendieck construction, $(\infty,2)$-category.
  	\par\vskip\baselineskip\noindent
  	\textbf{MSC}: 18N65, 18N40, 18N99.
  \end{abstract}
  
  \tableofcontents
  
 \section{Introduction}
 \subsection{Partially lax (co)limits}
 As one tries to generalize the 1-categorical notion of colimits to 2-categories one runs into an immediate problem: which definition of colimit to use. Loosely speaking, any definition of a colimit should come equipped with a universal cone. However, if we consider a  2-functor $F:\CC\to \DD$, we run into an issue defining cones over $F$. A cone over $F$ with tip $d$ should consist of:
  \begin{itemize}
    \item For every object $c\in \CC$, a morphism $\alpha_c:F(c)\to d$.
    \item For every morphism $u:b\to c$ in $\CC$, a diagram 
    \[
    \begin{tikzcd}
    F(b)\arrow[dd,"F(u)"'] \arrow[dr,"\alpha_b"] & \\
    & d\\
    F(c)\arrow[ur,"\alpha_c"']
    \end{tikzcd}
    \]
    that commutes appropriately.
  \end{itemize}
  It is here that the definition flounders --- there are multiple 2-categorical notions which could be described as the diagram “commuting appropriately“, and each yields different notions of colimit. If one requires the triangles to commute up to non-invertible 2-morphism, for instance, one obtains the notion of a \emph{lax} colimit. If, on the other had, one requires commutativity up to invertible 2-morphism, the corresponding notion of colimit is the \emph{pseudo-colimit}. 

  One traditional way of resolving the multiplicity of definitions of 2-dimensional colimits is by defining the more general notion of \emph{weighted colimits}, which specialize to each of the above cases (see for example \cite{Kelly}). However, in the past years, a different (but equivalent) approach  has become relevant due to its amenability to applications in simplicial models for higher categories: \emph{partially lax colimits}. 

  In defining partially lax colimits, one considers an $(\infty,2)$-category $\bcat{C}$ equipped with a collection of \emph{marked} 1-morphisms (which we usually denote with a subscript notation $\bcat{C}^\dagger$), and then requires that the chosen 2-morphism making the triangle above commute is invertible whenever $u$ is a marked morphism. This resolution of the above issue loses nothing in comparison to $\Cat$-weighted or $\bcat{C}\!\on{at}_\infty$-weighted limits, as the two theories turn out to be equivalent (see \cite[Theorem 4.7]{AG_cof} and \cite[Section 5]{GHL_LaxLim} for more details). Although this definition of 2-categorical limit is in fact a novel concept in the study of $\infty$-category theory its use in the strict 2-categorical realm was already established as in seen in \cite{Dubuc_limits}.

  Before continuing with our general discussion we present some examples. Let $\Lambda^2_2$ be the poset consisting in three objects $0,1,2$ and morphisms $1 \to 2$ and $0 \to 2$ . Suppose we are given a diagram
\vskip 1em
 \[
  \begin{tikzcd}
  \scr{A} \arrow[rrd,"F"] &  &         &              &    &                         &              \\
                      &  & \scr{C} & {} \arrow[squiggly,r] & {}  \arrow[phantom,l]& D:\Lambda^2_2 \arrow[r] & \bcat{C}\!\on{at}_\infty \\
  \scr{B} \arrow[rru,"G"] &  &         &              &    &                         &             
  \end{tikzcd}
 \]
  and let us compute the lax limit of $D$. We informally describe the lax limit  which we denote by $\scr{A} \times^\flat_{\scr{C}}\scr{B}$ as follows:
  \begin{itemize}
    \item Objects are given the following data: A triple of objects $a \in \scr{A}$, $b \in \scr{B}$ and $c \in \C$ together with morphisms $\alpha_a:F(a) \to c$ and $\alpha_b:G(b) \to c$.
    \item A morphism between $(a,b,c,\alpha_a,\alpha_b) \to (a',b',c',\alpha_{a'},\alpha_{b'})$ is given by morphisms $a \to a'$, $b \to b'$ and $c \to c'$ and a commutative diagram in $\C$
    \[
      \begin{tikzcd}
        F(a) \arrow[d] \arrow[rr,"\alpha_a"] &  & c \arrow[d] &  & G(b) \arrow[ll,"\alpha_b",swap] \arrow[d] \\
        F(a') \arrow[rr,"\alpha_{a'}"]          &  & c'          &  & G(b') \arrow[ll,"\alpha_{b'}",swap] .        
      \end{tikzcd}
    \]
  \end{itemize}
  This construction is known as the lax pullback and has been used in \cite{Exci} to understand the failure of excision for any localising invariant. The pseudolimit $\scr{A} \times^\sharp_{\scr{C}}\scr{B}$ (which coincides with the usual $\infty$-categorical pullback) is the full subcategory of $\scr{A} \times^\flat_{\scr{C}}\scr{B}$ on those tuples $(a,b,c,\alpha_a,\alpha_b)$ such that both $\alpha_a$ and $\alpha_b$ are equivalences.

 Now, let us suppose that $\Lambda^2_2$ comes equipped with a marking consisting on the edge $0 \to 2$ which corresponds to the functor $G:\scr{B} \to \scr{C}$. In this case we denote the partially lax limit by $\scr{A}\overset{\to}{\times}_\C \scr{B}$. One can show that the marking-depending limit is given by the subcategory consisting in tuples $(a,b,c,\alpha_a,\alpha_b)$ such that $\alpha_b$ is an equivalence in $\C$. In this situation we can describe the $\infty$-category  $\scr{A}\overset{\to}{\times}_\C \scr{B}$  as follows:
\begin{itemize}
  \item Objects are given by triples $(a,b,\alpha)$ where $a \in \scr{A}$, $b \in \scr{A}$ and $\alpha:F(a) \to G(b)$. 
  \item A morphism $(a,b,\alpha) \to (a',b',\alpha')$ is given by morphisms $a \to a'$, $b \to b'$ together with a commutative diagram
  \[
    \begin{tikzcd}
      F(a) \arrow[d] \arrow[r,"\alpha"] & F(b) \arrow[d] \\
      F(a') \arrow[r,"\alpha'"] & F(b')
    \end{tikzcd}
  \]
\end{itemize}
One analogously defines the $\infty$-category $\scr{A}\overset{\xleftarrow{}}{\times}_\C \scr{B}$ which corresponds to the partially lax limit of $D:\Lambda^2_2 \to \bcat{C}\!\on{at}_\infty$ where we are marking the edge $0 \to 1$ in $\Lambda^2_2$.

 Even in simpler examples, the theory of (co)limits in $\infty$-bicategories is capable of capturing interesting phenomena. For example, given an exact functor of stable $\infty$-categories $F:\scr{A} \to \scr{B}$ viewed  as a diagram $E:\Delta^1 \to \mathbb{S}\!\on{t}$ with values in the $\infty$-bicategory of stable $\infty$-categories and exact functor, the lax limit of E is a stable $\infty$-category which corresponds to a semiorthogonal decomposition (see \cite{Bondal}) of $\scr{A}$ and $\scr{B}$ along the gluing functor $F$. In particular, this allows us to characterize a semiorthogonal decomposition of stable $\infty$-categories in terms of a 2-dimensional universal property.

 \subsection{The cofinality theorem}
Let $f:\bcat{C}^\dagger \to \bcat{D}^\dagger$ be a marking-preserving functor of $\infty$-bicategories. We say that $f$ is \emph{marked cofinal} if for every diagram $F: \bcat{D} \to \bcat{A}$ the canonical comparison map
\[
  \begin{tikzcd}[ampersand replacement=\&]
    \colim_{\bcat{C}}^\dagger F f \arrow[r,"\simeq"] \& \colim_{\bcat{D}}^\dagger F
  \end{tikzcd}
\]
is an equivalence in $\bcat{A}$.

The main theorem of this paper gives a complete characterization of cofinal functors. Let us introduce some preliminary notation to better understand our main result.

\begin{definition}\label{defn:upslice}
  Let $f:\bcat{C}^\dagger \to \bcat{D}^\dagger$ be a marking preserving functor of $\infty$-bicategories. Given an object $d \in \bcat{D}$ we define the marked comma $\infty$-bicategory $\bcat{C}^\dagger_{d\upslash}$  as follows:
  \begin{itemize}[noitemsep]
    \item Objects are given by pairs $(u,c)$ where $u:d \to f(c)$ is a morphism in $\bcat{D}$ with source $d$ and $c$ is an object of $\bcat{C}$. 
    \item A 1-morphism from $u: d \to f(c)$ to $v:d \to f(c')$ is given by a 1-morphism $\alpha: c \to c'$ in $\bcat{C}$ and a 2-morphism $f(\alpha) \circ u \xRightarrow{} v$.
    \item A 2-morphism in $\bcat{C}^\dagger_{d\upslash}$ is given by a 2-morphism $\epsilon: \alpha  \nat \beta$ such that the diagram below commutes
    \[
    \begin{tikzcd}
    f(\alpha) \circ u \arrow[dd, Rightarrow,"f(\epsilon)*u",swap] \arrow[rrd, Rightarrow] &  &   \\
    &  & v \\
    f(\beta) \circ u \arrow[rru, Rightarrow]                        &  &  
    \end{tikzcd}
    \]
    \item A morphism in $\bcat{C}^\dagger_{d\upslash}$ is marked precisely when $\alpha: c \to c'$ is marked in $\bcat{C}^\dagger$ and the associated 2-morphism is invertible.
  \end{itemize}
   If the functor $f$ is the identity on the marked $\infty$-bicategory $\bcat{D}^\dagger$ we will use the notation $\bcat{D}^\dagger_{d\upslash}$.
   \end{definition}
   We state our main theorem which can be found in \autoref{thm:cofinality}
   \begin{thm*}
    Let $f:\bcat{C}^\dagger \to \bcat{D}^\dagger$ be a marking-preserving functor of $\infty$-bicategories. Then the following statements are equivalent:
  \begin{enumerate}
    \item[i)] The functor $f$ is marked cofinal.
    \item[ii)] For every $d \in \bcat{D}$ the functor $f$ induces an equivalence of $\infty$-categorical localizations $L_W(\bcat{C}_{d\upslash}^\dagger) \to L_W(\bcat{D}_{d\upslash}^\dagger)$.
  \end{enumerate}
  \end{thm*}
  We would like to point out that in \autoref{thm:cofinality} we give an equivalent set of conditions to that of $i)$ and $ii)$ above which are of more computational nature. 
  
  First, let us derive some corollaries from the previous theorem. Let us suppose that $\bcat{C}^\dagger=\C^{\sharp}$, $\bcat{D}^{\dagger}=\D^\sharp$, that is, both $\infty$-bicategories are actually $\infty$-categories with all morphisms being marked. Then since for every $d \in \D$ the $\infty$-category  $\D_{d/}$ has an initial object it follows that the $\infty$-categorical localization at all morphisms (which is precisely given by the geometric realization) must be a contractible space $L_W(\D_{d/}^\sharp) \isom |\D_{d/}|\isom *$. Then the second statement in our theorem collapses to:
\begin{itemize}
  \item For every $d \in \D$, the geometric realization of the comma category $|\C_{d/}|\isom *$ is contractible. 
\end{itemize}
In order words, our theorem recovers the characterization of cofinal functors of $\infty$-categories due to Joyal. In a similar way as the original theorem of Quillen can be recovered from the main cofinality statement, in our situation we can obtain the following generalization of Quillen's Theorem A.
\begin{cor*}
  Let $f:\bcat{C}^\dagger \to \bcat{D}^\dagger$ be a marking-preserving functor of $\infty$-bicategories and suppose that the following condition holds
  \begin{itemize}
    \item For every $d \in \bcat{D}$ the functor $f$ induces an equivalence of $\infty$-categorical localizations $L_W(\bcat{C}_{d\upslash}^\dagger) \to L_W(\bcat{D}_{d\upslash}^\dagger)$.
  \end{itemize}
  Then $f$ induces an equivalence upon passage to $\infty$-categorical localizations
  \[
    \begin{tikzcd}
      L_W(f): L_W\left(\bcat{C}^\dagger\right) \arrow[r,"\simeq"] & L_W(\bcat{D}^\dagger).
    \end{tikzcd}
  \]
\end{cor*}

The proof of our main theorem relies on the fact (see the last section in \cite{GHL_LaxLim}) that a map of marked $\infty$-bicategories $f:\bcat{C}^\dagger \to \bcat{D}^\dagger$ is cofinal if and only if it defines an  outer Cartesian equivalence over $\bcat{D}$. Our proof strategy consists in showing that $\bcat{C}^\dagger \to \bcat{D}$ is classified (after pertinent fibrant replacements) by the functor 
\[
  L_W\left(\bcat{C}^\dagger_{\bcat{D}\upslash}\right): \bcat{D}^\op \to \bcat{C}\!\on{at}_\infty, \enspace \enspace d \mapsto L_W\left(\bcat{C}^\dagger_{d \upslash}\right)
 \] 
 and similarly $\bcat{D}^\dagger \to \bcat{D}$ is classified by the functor 
\[
  L_W\left(\bcat{D}^\dagger_{\bcat{D}\upslash}\right): \bcat{D}^\op \to \bcat{C}\!\on{at}_\infty, \enspace \enspace d \mapsto L_W\left(\bcat{D}^\dagger_{d \upslash}\right)
 \] 
 so that we can identify the map $f:\bcat{C}^\dagger \to \bcat{D}^\dagger$ with its associated natural transformation  which has as components the canonical maps 
 \[
    L_W\left(\bcat{C}^\dagger_{d \upslash}\right) \to  L_W\left(\bcat{D}^\dagger_{d \upslash}\right)
 \]
This argument makes heavy use of the $\infty$-bicategorical Grothendieck construction studied by the authors in \cite{AGS_CartII}. Indeed, in order to gain access to the functor $ L_W\left(\bcat{C}^\dagger_{\bcat{D}\upslash}\right)$ we construct for every functor of $\infty$-bicategories $f: \bcat{C} \to \bcat{D}$ an outer 2-Cartesian fibration $\mathbb{F}(\bcat{C}) \to \bcat{D}$ which corresponds under our Grothendieck construction to the functor
\[
  \bcat{C}_{\bcat{D}\upslash}: \bcat{D}^\op \to \bcat{B}\!\on{icat}_\infty, \enspace d \mapsto \bcat{C}^\natural_{d\upslash}
\]
where the marking of $\bcat{C}^\natural$ consists in the equivalences. Our theorem, follows from a simple localization argument together with the next result which identifies $F(\bcat{C}) \to \bcat{D}$ as the free 2-Cartesian fibration on the functor $f$.

\begin{theorem}
  Let $f:\bcat{C} \to \bcat{D}$ be a functor of $\infty$-bicategories. Then there is a morphism over $\bcat{D}$ 
  \[
    \begin{tikzcd}
    \bcat{C} \arrow[dr,"f"] \arrow[rr] && \mathbb{F}(\bcat{C}) \arrow[ld] \\
    & \bcat{D} &
  \end{tikzcd}
  \]
  which is an equivalence in the model structure for outer 2-Cartesian fibrations defined in \cite[\S 3]{AGS_CartI}.
\end{theorem}

\subsection{A class of examples}

We conclude this section with a somewhat general example that illustrates the applicability of our theorem.

Let $K$ be an ordinary category. We define a 2-category $\mathbb{2}[K]$ as follows:
\begin{itemize}
   \item We have a pair of objects $0,1$.
   \item The mapping categories are given by $\mathbb{2}[K](\epsilon,\epsilon)=\Delta^0$ for $\epsilon \in \set{0,1}$, $\mathbb{2}[K](0,1)=K$ and $\mathbb{2}[K](1,0)=\emptyset$.  
 \end{itemize}
Given a functor of 1-categories $p:K \to S$ we obtain a (strict) 2-functor $\mathbb{p}:\mathbb{2}[K] \to \mathbb{2}[S]$. We will study the morphisms
\[
  \begin{tikzcd}
    \mathbbm{p}_i: \mathbb{2}[K]_{i\upslash} \arrow[r] & \mathbb{2}[S]_{i \upslash}, \enspace \text{ for }i=0,1.
  \end{tikzcd}
\]
The case $i=1$ is trivial since the functor $\mathbbm{p}_1$ is an isomorphism. We turn our attention to the case $i=0$. Observe that since $\mathbbm{p}$ is an isomorphism on objects it follows that both  $\mathbb{2}[K]_{0\upslash}$ and $ \mathbb{2}[S]_{0\upslash}$ have the same objects which are precisely represented by objects $s \in S$ together with the identity morphism on the object $0$. Let $s_1,s_2 \in S$. Then $\mathbbm{p}_0$ induces an isomorphism 
\[
  \mathbb{2}[K]_{0\upslash}(s_1,s_2)\isom \mathbb{2}[S]_{0\upslash}(s_1,s_2)\isom \Hom_S(s_1,s_2)
\]
We similarly obtain
\[
   \mathbb{2}[K]_{0\upslash}(s,\on{id})\isom  \mathbb{2}[S]_{0\upslash}(s,\on{id})\isom \emptyset, \enspace \enspace \mathbb{2}[K]_{0\upslash}(\on{id},\on{id})\isom  \mathbb{2}[S]_{0\upslash}(\on{id},\on{id})\isom \Delta^0
\]
The final case to analyze then yields
\[
  \func{K_{/s}=\mathbb{2}[K]_{0\upslash}(\on{id},s) \to \mathbb{2}[S]_{0\upslash}(\on{id},s)=S_{/s}}
\]
where $K_{/s}$ is the dual version\footnote{These categories control the dual notion of coinitial functor: A functor is $f$ is coinitial if restriction along $f$ preserves limits } of $K_{s/}$ whose objects are given by morphisms $u:p(k) \to s$ in $S$ and whose morphisms are given by maps $k \to k'$ making the obvious diagram commute. We can now use our characterization of cofinal functors of $\infty$-bicategories to arrive at the following result.
\begin{prop*}
  Let $p:K \to S$ be a functor of ordinary 1-categories. Then $\mathbbm{p}:\mathbb{2}[K] \to \mathbb{2}[S]$ is a cofinal functor of 2-categories (with respect to the minimal marking) if and only if $p^\op:K^\op \to S^\op$ is a cofinal functor of 1-categories.
\end{prop*}

Using this preposition we can provide natural examples of cofinal maps
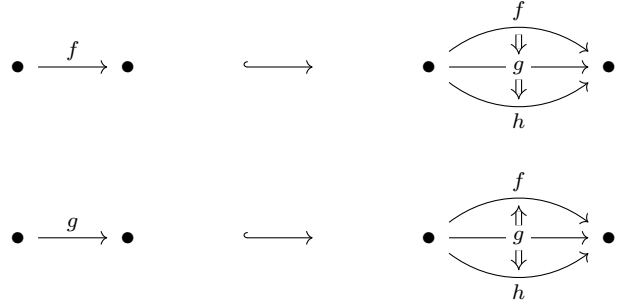
\[\begin{tikzcd}
  \bullet & \bullet & {} & {} & \bullet && \bullet
  \arrow[""{name=0, anchor=center, inner sep=0}, "f", curve={height=-18pt}, from=1-5, to=1-7]
  \arrow[""{name=1, anchor=center, inner sep=0}, "h"', curve={height=18pt}, from=1-5, to=1-7]
  \arrow[""{name=2, anchor=center, inner sep=0}, "g"{description}, from=1-5, to=1-7]
  \arrow["f", from=1-1, to=1-2]
  \arrow[hook, from=1-3, to=1-4]
  \arrow[shorten <=2pt, shorten >=4pt, Rightarrow, from=0, to=2]
  \arrow[shorten <=4pt, shorten >=2pt, Rightarrow, from=2, to=1]
\end{tikzcd}\]
\[\begin{tikzcd}
  \bullet & \bullet & {} & {} & \bullet && \bullet
  \arrow[""{name=0, anchor=center, inner sep=0}, "f", curve={height=-18pt}, from=1-5, to=1-7]
  \arrow[""{name=1, anchor=center, inner sep=0}, "h"', curve={height=18pt}, from=1-5, to=1-7]
  \arrow[""{name=2, anchor=center, inner sep=0}, "g"{description}, from=1-5, to=1-7]
  \arrow["g", from=1-1, to=1-2]
  \arrow[hook, from=1-3, to=1-4]
  \arrow[shorten <=2pt, shorten >=4pt, Rightarrow, 2tail reversed, no head, from=0, to=2]
  \arrow[shorten <=4pt, shorten >=2pt, Rightarrow, from=2, to=1]
\end{tikzcd}\]

\subsection{Variances}

In the theory of partially lax (co)limits, as expounded in \cite{GHL_LaxLim}, there are in fact four versions to consider. In classical terminology, these correspond to lax and oplax limits and colimits. These notions are related by the 1- and 2-morphism duals of infinity bicategories, and as such, our cofinality criterion provides formally dual characterizations of when the restriction along a functor $\func{F:\mathbb{C}\to \mathbb{D}}$ of $\infty$-bicategories preserves partially (op)lax (co)limits. That is, all variances can be addressed by applying our criterion to an appropriate formal dual. 

Moreover, each of the four variances possible for (op)lax (co)limits corresponds to a variance of the $(\infty,2)$-categorical Grothendieck construction. In this paper we work with the Grothendieck construction we developed in \cite{AGS_CartII}, and, correspondingly with lax colimits. Since this Grothendieck construction also, via opposite simplicial sets, provides a second, this gives explicit constructions for two of the four variances immediately.

Our arguments, however, generalize to the other variances nearly verbatim, requiring only that where applicable one reverse the orders of Gray tensor products and swap Gray functor categories (those with lax transformations as their morphisms) for opGray functor categories (those with oplax transformations as their morphisms). The upcoming work \cite{AG_dualfib} of the first-named author will provide a model-categorical construction of the two remaining Grothendieck constructions, which can be used to make the dual argument more explicit.

In the special case when we take partially lax colimits over a strict 2-category with a marking, the criterion of the corollary above can be checked by working with the lax slices computed using strict 2-categorical methods, rather than the $\infty$-bicategorical version, as we show in \autoref{thm:comparisonofslices}. Of particular import, this means that one can check the corresponding notions of cofinality for the other variances by taking the appropriate dual of the \emph{2-categories} in question, and then applying the criterion of the corollary. Thus, in practice, whenever we index over a strict 2-category, the computation criterion of the corollary is very tractable regardless of variance. 

\section{Preliminaries}

Recapitulating all the necessary background in the theory of $(\infty,2)$-categories and their colimits would render this preliminary section longer than the rest of the paper. We have therefor decide to include only the relevant results from out previous works \cite{AGS_CartI} and \cite{AGS_CartII} that play a central role in our arguments, that is:
\begin{itemize}
   \item The theory of marked biscaled simplicial sets and the model structure for outer 2-Cartesian fibrations.
   \item The $\infty$-bicategorical Grothendieck construction. 
 \end{itemize} 

 We refer the reader to \cite{LurieGoodwillie} for an extensive study of scaled simplicial sets and the corresponding model structures. For the theory of partially lax (co)limits, we refer the reader to \cite{GHL_LaxLim}.

 \subsection{Marked biscaled simplicial sets and outer 2-Cartesian fibrations.}

\begin{definition}
  A \emph{marked biscaled} simplicial set ($\bS$ simplicial set) is given by the following data
  \begin{itemize}
    \item A simplicial set $X$.
    \item A collection of edges  $E_X \in X_1$ containing all degenerate edges.
    \item A collection of triangles $T_X \in X_2$ containing all degenerate triangles. We will refer to the elements of this collection as \emph{thin triangles}.
    \item A collection of triangles $C_X \in X_2$ such that $T_X \subseteq C_X$. We will refer to the elements of this collection as \emph{lean triangles}.
  \end{itemize}
  We will denote such objects as triples $(X,E_X, T_X \subseteq C_X)$. A map $(X,E_X, T_X \subseteq C_X) \to (Y,E_Y,T_Y \subseteq C_Y)$ is given by a map of simplicial sets $f:X \to Y$ compatible with the collections of edges and triangles above. We denote by $\mbsSet$ the category of $\bS$ simplicial sets. 
\end{definition}

\begin{notation}
  Let $(X,E_X, T_X \subseteq C_X)$ be a $\bS$ simplicial set. Suppose that the collection $E_X$ consist only of degenerate edges. Then we fix the notation $(X,E_X, T_X \subseteq C_X)=(X,\flat,T_X \subseteq E_X)$ and do similarly for the collection $T_X$. If $C_X$ consists only of degenerate triangles we fix the notation $(X,E_X, T_X \subseteq C_X)=(X,E_X, \flat)$. In an analogous fashion we wil use the symbol “$\sharp$“ to denote a collection containing all edges (resp. all triangles). Finally suppose that $T_X=C_X$ then we will employ the notation $(X,E_X,T_X)$.
\end{notation}

\begin{remark}
  We will often abuse notation when defining the collections $E_X$ (resp. $T_X$, resp. $C_X$) and just specified its non-degenerate edges (resp. triangles).
\end{remark}

\begin{definition}\label{def:mbsanodyne}
  The set of \emph{generating $\bS$-anodyne maps}, \(\bS\) is the set of maps of $\bS$ simplicial sets consisting of:
  \begin{enumerate}
    \myitem{(A1)}\label{mb:innerhorn} The inner horn inclusions 
    \[
    \bigl(\Lambda^n_i,\flat,\{\Delta^{\{i-1,i,i+1\}}\}\bigr)\rightarrow \bigl(\Delta^n,\flat,\{\Delta^{\{i-1,i,i+1\}}\}\bigr)
    \quad , \quad n \geq 2 \quad , \quad 0 < i < n ;
    \]
    \myitem{(A2)}\label{mb:wonky4} The map 
    \[
    (\Delta^4,\flat,T)\rightarrow (\Delta^4,\flat,T\cup \{\Delta^{\{0,3,4\}}, \ \Delta^{\{0,1,4\}}\}),
    \]
    where we define
    \[
    T\overset{\text{def}}{=}\{\Delta^{\{0,2,4\}}, \ \Delta^{\{ 1,2,3\}}, \ \Delta^{\{0,1,3\}}, \ \Delta^{\{1,3,4\}}, \ \Delta^{\{0,1,2\}}\};
    \]
    \myitem{(A3)}\label{mb:leftdeglefthorn} The set of maps
    \[
    \Bigl(\Lambda^n_0\coprod_{\Delta^{\{0,1\}}}\Delta^0,\flat,\flat \subset\{\Delta^{\{0,1,n\}}\}\Bigr)\rightarrow \Bigl(\Delta^n\coprod_{\Delta^{\{0,1\}}}\Delta^0,\flat,\flat \subset\{\Delta^{\{0,1,n\}}\}\Bigr)
    \quad , \quad n\geq 2.
    \]
    These maps force left-degenerate lean-scaled triangles to represent coCartesian edges of the mapping category.
    \myitem{(A4)}\label{mb:2Cartesianmorphs} The set of maps
    \[
    \Bigl(\Lambda^n_n,\{\Delta^{\{n-1,n\}}\},\flat \subset \{ \Delta^{\{0,n-1,n\}} \}\Bigr) \to \Bigl(\Delta^n,\{\Delta^{\{n-1,n\}}\},\flat \subset \{ \Delta^{\{0,n-1,n\}} \}\Bigr) \quad , \quad n \geq 2.
    \]
    This forces the marked morphisms to be $p$-Cartesian with respect to the given thin and lean triangles. 
    \myitem{(A5)}\label{mb:2CartliftsExist} The inclusion of the terminal vertex
    \[
    \Bigl(\Delta^{0},\sharp,\sharp \Bigr) \rightarrow \Bigl(\Delta^1,\sharp,\sharp \Bigr).
    \]
    This requires $p$-Cartesian lifts of morphisms in the base to exist.
    \myitem{(S1)}\label{mb:composeacrossthin} The map
    \[
    \Bigl(\Delta^2,\{\Delta^{\{0,1\}}, \Delta^{\{1,2\}}\},\sharp \Bigr) \rightarrow \Bigl(\Delta^2,\sharp,\sharp \Bigr),
    \]
    requiring that $p$-Cartesian morphisms compose across thin triangles.
    \myitem{(S2)}\label{mb:coCartoverThin} The map
    \[
    \Bigl(\Delta^2,\flat,\flat \subset \sharp \Bigr) \rightarrow \Bigl( \Delta^2,\flat,\sharp\Bigr),
    \]
    which requires that lean triangles over thin triangles are, themselves, thin.
    \myitem{(S3)}\label{mb:innersaturation} The map
    \[
    \Bigl(\Delta^3,\flat,\{\Delta^{\{i-1,i,i+1\}}\}\subset U_i\Bigr) \rightarrow \Bigl(\Delta^3,\flat, \{\Delta^{\{i-1,i,i+1\}}\}\subset \sharp \Bigr) \quad, \quad 0<i<3
    \]
    where $U_i$ is the collection of all triangles except $i$-th face. This and the next two generators serve to establish composability and limited 2-out-of-3 properties for lean triangles.
    \myitem{(S4)}\label{mb:dualcocart2of3} The map
    \[
    \Bigl(\Delta^3 \coprod_{\Delta^{\{0,1\}}}\Delta^0,\flat,\flat \subset U_0\Bigr) \rightarrow \Bigl(\Delta^3 \coprod_{\Delta^{\{0,1\}}}\Delta^0,\flat, \flat \subset \sharp \Bigr) 
    \]
    where $U_0$ is the collection of all triangles except the $0$-th face.
    \myitem{(S5)}\label{mb:coCart2of3} The map
    \[
    \Bigl(\Delta^3,\{\Delta^{\{2,3\}}\},\flat \subset U_3\Bigr) \rightarrow \Bigl(\Delta^3,\{\Delta^{\{2,3\}}\}, \flat \subset \sharp \Bigr) 
    \]
    where $U_3$ is the collections of all triangles except the $3$-rd face.
    \myitem{(E)}\label{mb:equivalences} For every Kan complex $K$, the map
    \[
    \Bigl( K,\flat,\sharp  \Bigr) \rightarrow \Bigl(K,\sharp, \sharp\Bigr).
    \]
    Which requires that every equivalence is a marked morphism.
  \end{enumerate}
  A map of $\bS$ simplicial sets is said to be \bS-anodyne if it belongs to the weakly saturated closure of \bS.
\end{definition}

\begin{definition}
  Let $f:(X,E_X,T_X \subseteq C_X) \to (Y,E_Y,T_Y \subseteq C_Y)$ be a map of $\bS$ simplicial sets. We say that $f$ is a \bS-fibration if it has the right lifting property against the class of \bS-anodyne morphisms.
\end{definition}

\begin{definition}\label{def:mappingbicats}
  Given two $\bS$ simplicial sets $(K,E_K,T_K \subseteq C_K), (X,E_X,T_X \subseteq C_X)$ we define  another $\bS$ simplicial set denoted by $\on{Fun}^{\mathbf{mb}}(K,X)$ and characterized by the following universal property
  \[
  \on{Hom}_{\mbsSet}\Bigr(A,\on{Fun}^{\mathbf{mb}}(K,X) \Bigl)\isom \Hom_{\mbsSet}\Bigr(A \times K,X  \Bigl).
  \]
\end{definition}

\begin{proposition}\label{prop:bsfibfun}
  Let $f:(X,E_X,T_X \subseteq C_X) \to (Y,E_Y,T_Y \subseteq C_Y)$ be a \bS-fibration. Then for every $K \in \mbsSet$ the induced morphism $\on{Fun}^{\mathbf{mb}}(K,X) \to \on{Fun}^{\mathbf{mb}}(K,Y)$ is a \bS-fibration.
\end{proposition}

\begin{definition}
  Let $f: Y \to S$ be a morphism of $\bS$ simplicial another map $g:X \to Y$. We define a $\bS$ simplicial set $\on{Map}_Y(K,X)$ by means of the pullback square
  \[
  \begin{tikzcd}[ampersand replacement=\&]
    \on{Map}_Y(K,X) \arrow[r] \arrow[d] \& \on{Fun}^{\mathbf{mb}}(K,X) \arrow[d] \\
    \Delta^0 \arrow[r,"g"] \& \on{Fun}^{\mathbf{mb}}(K,Y)
  \end{tikzcd}
  \]
  If $f:X \to Y$ is a \bS-fibration then it follows from the previous proposition that $\on{Map}_Y(K,X)$ is an $\infty$-bicategory.
\end{definition}

Let $S \in \on{Set}^{\mathbf{sc}}_{\Delta}$ for the rest of the section we will denote $(\mbsSet)_{/S}$ the category of $\bS$ simplicial sets over $(S,\sharp,T_S \subset \sharp)$.

\begin{definition}\label{def:fibrantobjects}
  We say that an object $\pi:X \to S$ in $(\mbsSet)_{/S}$ is an \emph{outer} 2-\emph{Cartesian} fibration if it is a $\bS$-fibration.
\end{definition}

\begin{remark}
  We will frequently abuse notation and refer to outer 2-Cartesian as \emph{2-Cartesian fibrations}.
\end{remark}

\begin{definition}\label{def:underlyingmapping}
  Let $\pi:X \to S$ be a morphism of $\bS$ simplicial sets. Given an object $K\to S$, we define $\on{Map}^{\on{th}}_{S}(K,X)$ to be the  $\bS$ simplicial subset consisting only of the thin triangles. Note that if $\pi$ is a 2-Cartesian fibration this is precisely the underlying $\infty$-category of $\on{Map}_S(K,X)$. 
  
  We similarly denote by $\on{Map}^{\isom}_S(K,X)$ the mb sub-simplicial set consisting of thin triangles and marked edges. As before, we note that if $\pi$ is a 2-Cartesian fibration, the simplicial set $\on{Map}^{\isom}_S(K,X)$ can be identified with the maximal Kan complex in $\on{Map}_S(K,X)$.
\end{definition}

\begin{definition}\label{def:weakequiv1}
  Let $\func{L \to[h] K \to[p] S}$ be a morphism in $(\mbsSet)_{/S}$. We say that $h$ is a cofibration when it is a monomorphism of simplicial sets. We will call $h$ a weak equivalence if for every 2-Cartesian fibration $\pi:X \to S$ the induced morphism
  \[
  \func{h^{*}:\on{Map}_S(K,X) \to \on{Map}_S(L,X)}
  \]
  is a bicategorical equivalence.
\end{definition}

We can now recall the main results of \cite{AGS_CartI}.

\begin{theorem}[\cite{AGS_CartI} Theorem 3.38]\label{thm:MBModelStructure}
  Let $S$ be a scaled simplicial set. Then there exists a left proper combinatorial simplicial model structure on $(\mbsSet )_{/S}$, which is characterized uniquely by the following properties:
  \begin{itemize}
    \item[C)] A morphism $f:X \to Y$ in $(\mbsSet )_{/S}$ is a cofibration if and only if $f$ induces a monomorphism betwee the underlying simplicial sets.
    \item[F)] An object $X \in (\mbsSet )_{/S}$ is fibrant if and only if $X$ is a 2-Cartesian fibration. 
  \end{itemize}
\end{theorem}

\subsection{The \texorpdfstring{$\infty$-}-bicategorical Grothendieck construction.}
 A special case of the model structure of \autoref{thm:MBModelStructure} of particular interest occurs when $S=\Delta^0$ is the terminal scaled simplicial set. Then, by \cite[Thm 3.39]{AGS_CartI}, the resulting model structure on $\mbsSet$ is Quillen equivalent to the model structure for $\infty$-bicategories on $\Set_\Delta^{\mathbf{sc}}$. In this case, the data of the two scalings becomes highly redundant --- for any fibrant object the two scalings coincide, and heuristically they no longer encode different information. 

We can avoid this redundancy by defining a further model structure which includes both markings and scalings, but avoids the redundancies created by a biscaling. The aim of this section is to define this model structure, and relate it to the \textbf{MB} model structure. 

\begin{definition}
  A \emph{marked-scaled simplicial set} consists of 
  \begin{itemize}
    \item A simplicial set $X$.
    \item A collection of edges $E_X\subseteq X_1$ containing all degenerate edges. We call the elements of $E_X$ \emph{marked edges}.
    \item A collection of triangles $T_X\subseteq X_2$ containing all degenerate triangles. We call the elements of $T_X$ \emph{thin triangles}.
  \end{itemize}
 \end{definition} 

 \begin{definition}
  The set of \emph{generating $\mathbf{MS}$-anodyne maps} $\mathbf{MS}$ is the set of maps of marked-scaled simplicial sets consisting of:
  \begin{itemize}
    \myitem{(MS1)}\label{MS:inner} The inner horn inclusions 
    \[
    \bigl(\Lambda^n_i,\flat,\{\Delta^{\{i-1,i,i+1\}}\}\bigr)\rightarrow \bigl(\Delta^n,\flat,\{\Delta^{\{i-1,i,i+1\}}\}\bigr)
    \quad , \quad n \geq 2 \quad , \quad 0 < i < n ;
    \]
    \myitem{(MS2)}\label{MS:wonky4} The map 
    \[
    \func{
      (\Delta^4,\flat,T) \to (\Delta^4,\flat, T\cup \{\Delta^{\{0,3,4\}},\Delta^{\{0,1,4\}}\})
    }
    \]
    where $T$ is defined as in \autoref{def:mbsanodyne}, \ref{mb:wonky4}.
    \myitem{(MS3)}\label{MS:0horn} The set of maps 
    \[
    \Bigl(\Lambda^n_0\coprod_{\Delta^{\{0,1\}}}\Delta^0,\flat,\{\Delta^{\{0,1,n\}}\}\Bigr)\rightarrow \Bigl(\Delta^n\coprod_{\Delta^{\{0,1\}}}\Delta^0,\flat,\{\Delta^{\{0,1,n\}}\}\Bigr)
    \quad , \quad n\geq 2.
    \]
    \myitem{(MS4)}\label{MS:nhorn} The set of maps 
    \[
    \Bigl(\Lambda^n_n,\{\Delta^{\{n-1,n\}}\}, \{ \Delta^{\{0,n-1,n\}} \}\Bigr) \to \Bigl(\Delta^n,\{\Delta^{\{n-1,n\}}\}, \{ \Delta^{\{0,n-1,n\}} \}\Bigr) \quad , \quad n \geq 2.
    \]
    \myitem{(MS5)}\label{MS:Cartlifts} The inclusion of the terminal vertex 
    \[
    \func{
      \left(\Delta^0,\sharp,\sharp\right)\to \left(\Delta^1,\sharp,\sharp\right)
    }
    \] 
    \myitem{(MS6)}\label{MS:Compose} The map 
    \[
    \Bigl(\Delta^2,\{\Delta^{\{0,1\}}, \Delta^{\{1,2\}}\},\sharp \Bigr) \rightarrow \Bigl(\Delta^2,\sharp,\sharp \Bigr),
    \]
    \myitem{(MS7)}\label{MS:composedeg4} The map
    \[
    \Bigl(\Delta^3 \coprod_{\Delta^{\{0,1\}}}\Delta^0,\flat, U_0\Bigr) \rightarrow \Bigl(\Delta^3 \coprod_{\Delta^{\{0,1\}}}\Delta^0,\flat, \sharp \Bigr) 
    \]
    where $U_0$ is the collection of all triangles except the $0$-th face.
    \myitem{(MS8)}\label{MS:composemarked5} The map
    \[
    \Bigl(\Delta^3,\{\Delta^{\{2,3\}}\}, U_3\Bigr) \rightarrow \Bigl(\Delta^3,\{\Delta^{\{2,3\}}\}, \sharp \Bigr) 
    \]
    where $U_3$ is the collections of all triangles except the $3$-rd face.
    \myitem{(MSE)}\label{MS:kan} For every Kan complex $K$, the map
    \[
    \Bigl( K,\flat,\sharp  \Bigr) \rightarrow \Bigl(K,\sharp, \sharp\Bigr).
    \]
  \end{itemize}
  We will call a morphism in $\mssSet$ \emph{$\mathbf{MS}$-anodyne} if it lies in the saturated hull of $\mathbf{MS}$. 
\end{definition}

\begin{lemma}\label{lem:Ui_MS-anodyne}
  The morphism
  \[
  \Bigl(\Delta^3,\flat,\{\Delta^{\{i-1,i,i+1\}}\}\subset U_i\Bigr) \rightarrow \Bigl(\Delta^3,\flat, \{\Delta^{\{i-1,i,i+1\}}\}\subset \sharp \Bigr) \quad, \quad 0<i<3,
  \]
  where $U_i$ is the collection of all triangles except $i$-th face, is $\mathbf{MS}$-anodyne.
\end{lemma}

\begin{proof}
  See \cite[Rmk 3.1.4]{LurieGoodwillie}.
\end{proof} 

In a similar way, as done for the $\bS$ case we can prove the following theorem:

\begin{theorem}[\cite{AGS_CartII} Theorem 2.45]\label{thm:markedscaledmodel}
  There is a left-proper combinatorial simplicial model category structure on $\mssSet$ uniquely characterized by the following properties:
  \begin{itemize}
    \item[C)] A morphism $f:X\to Y$ in $\mssSet$ is a cofibration if and only if it is a monomorphism on underlying simplicial sets.
    \item[F)] An object $X\in \mssSet$ is fibrant if and only if the unique map $X\to \Delta^0$ has the right lifting property with respect to the morphisms in $\mathbf{MS}$.  
  \end{itemize}
\end{theorem} 

We further prove in \cite[Lemma 2.49]{AGS_CartII} that this model structure is equivalent to the model structure on scaled simplicial sets. This model structure plays a fundamental role in our main theorem in \cite{AGS_CartII}

 \begin{theorem}[\cite{AGS_CartII} Theorem 3.85]
    Let $S$ be an scaled simplicial set. Then the bicategorical straightening-unstraightening adjunction defines a Quillen equivalence 
    \[
    \ST_S: (\mbsSet)_{/S} \llra \left(\mssSet\right)^{\mathfrak{C}[S]^\op}:\UN_S
    \]
    between the model structure on (outer) 2-Cartesian fibrations over $S$ and the projective model structure on $\msSet$-enriched functors $\mathfrak{C}[S]^\op \to \mssSet$ with values in marked-scaled simplicial sets.
  \end{theorem} 

 We further show that both model categories are in fact $\on{Set}^+_\Delta$-enriched categories. After performing elementary explicit verifications we prove that the functor $\UN_S$ is compatible with the (co)tensoring yielding an upgrade of the previous theorem to an intrinsinc bicategorical result.
  
  \begin{theorem}[\cite{AGS_CartII} Corollary 3.90]
    The bicategorical straightening is a left Quillen equivalence for any scaled simplicial set $S$. Moreover, the functor $\UN_S$ provides an equivalence of $(\infty,2)$-categories 
    \[
    2\bcat{C}\!\on{art}(S)\simeq \Fun(S^\op,\bcat{B}\!\on{icat}_\infty).
    \]
  \end{theorem}

\section{The main theorem}
In previous work \cite{AG_cof}, \cite{AGSQuillen} we have studied 2-categorical notions of (co)limits in $\infty$-bicategories. In this section we will prove the cofinality conjecture as stated in \cite{AGSQuillen}. Let us remark that in previous work we only considered specific instances of colimits: 
\begin{itemize}
  \item In \cite{AGSQuillen} we considered diagrams $F:\mathbb{D}^\dagger \to \mathbb{A}$ where both $\DD$ and $\mathbb{A}$ are strict 2-categories and $\DD^\dagger$ comes equipped with a collection of marked edges.
  \item In \cite{AG_cof} we considered diagrams $F:\scr{D}^\dagger \to \bcat{A}$ where $\D^\dagger$ is an $\infty$-category equipped with a collection of marked edges and $\bcat{A}$ is an $\infty$-bicategory.
\end{itemize}
The general notion of partially lax (co)limit has been extensively studied by Gagna, Harpaz and Lanari in \cite{GHL_LaxLim}. It was noted by the authors that their notion coincides with ours in the cases we studied as shown in \cite[Remark 5.2.12.]{GHL_LaxLim}. Moreover, the notion of cofinality studied in \cite{AG_cof} also agrees with the definition of cofinal functor given in \cite{GHL_LaxLim}. However, in the previous document no computational criterion is given to determine whether a functor of marked $\infty$-bicategories $f:\bcat{C}^\dagger \to \bcat{D}^\dagger$ is cofinal. We will provide such criterion as the main result in this section extending the well-known conditions of Quillen's Theorem A.

Recall that in \cite{GHL_LaxLim} the authors define in Section 4.1 a marked-scaled version of the Gray product. Given a marked scaled simplicial set $K$ and a $\infty$-bicategory $\bcat{D}$ it follows that we have a  $\infty$-bicategory $\on{Fun}^{\on{gr}}(K,\bcat{D})$ of  functors and partially lax natural transformations where the level of laxness is specified by the marking of $K$ (see Remark 4.1.13 in \cite{GHL_LaxLim}). Following their notation we will denote the mapping $\infty$-category of $\on{Fun}^{\on{gr}}(K,\bcat{D})$ by $\on{Nat}^{\on{gr}}_K(-,-)$.

\begin{remark}
  The theory of partially lax colimit comes in two variances: inner and outer colimits. In this document we will only study outer colimits which we will simply call (partially lax) colimits where the marking will be apparent in our notation.
\end{remark}

\begin{definition}
  Let $K^\dagger$ be a marked scaled simplicial set and let $F:K \to \bcat{A}$ be a functor of scaled simplicial sets where $K$ denotes the underlying scaled simplicial set of $K^\dagger$. The outer partially lax colimit of $F$ denoted by $\colimdag_K F$ is an object of $\bcat{A}$ corepresenting the functor
  \[
    \func{\on{Nat}^{\on{gr}}_K(F,c_K(\mathblank)): \bcat{A} \to \bcat{C}\!\on{at}_\infty; a \mapsto \on{Nat}^{\on{gr}}_K(F,\underline{a})}
  \]
  where $\underline{a}:K \to \bcat{A}$ is the constant functor with value $a \in \bcat{A}$.
\end{definition}

\subsection{The free 2-Cartesian fibration}\label{subsec:FreeFib}

Throughout this section, we fix an $\infty$-bicategory $\bcat{D}$, and aim to construct, for each functor of $\infty$-bicategories $p:\bcat{X}\to \bcat{D}$, a 2-Cartesian fibration $\mathbb{F}(p):\mathbb{F}(\bcat{X})\to \bcat{D}$. We will characterize this fibration as the \emph{free 2-Cartesian fibration on $p$.}

To construct this 2-Cartesian fibration, we will make use of the Gray tensor products constructed in \cite[\S 2.1]{GHL_Gray} and \cite[\S 4.1]{GHL_LaxLim}. To this end, we briefly recall the definitions we will need from \cite{GHL_Gray}.  

\begin{definition}
  Let $X,Y\in \scsSet$. We define the \emph{Gray product} $X\otimes Y$ to be the scaled simplicial set with underlying simplicial set $X\times Y$, where we declare a 2-simplex  $(\sigma_X,\sigma_Y)$ to be scaled if and only if the following two conditions are both satisfied
  \begin{itemize}
    \item $(\sigma_X,\sigma_Y)\in T_X\times T_Y$. 
    \item Either the image of $\sigma_X$  degenerates along $\Delta^{\{1,2\}}$ or $\sigma_Y$
    degenerates along $\Delta^{\{0,1\}}$.
  \end{itemize}
  Given scaled simplicial sets $X$ and $Y$, we define the \emph{Gray functor category} $\Fun^{\on{gr}}(X,Y)$ by the adjunction 
  \[
  \Hom_{\scsSet}(S,\Fun^{\on{gr}}(X,Y))\cong \Hom_{\scsSet}(X\otimes S,Y).
  \] 
  There is a dual version, defined by replacing $X\otimes S$ by $S\otimes X$, which we denote by $\Fun^{\on{opgr}}(X,Y)$. 
\end{definition}

\begin{proposition}
  Let $\bcat{D}$ an $\infty$-bicategory. Then the map $\func{\on{ev}_0:\Fun^{\on{gr}}(\Delta^1,\bcat{D}) \to \bcat{D}}$ is a 2-Cartesian fibration. The collection of Cartesian edges, thin triangles, and lean (coCartesian) triangles can be described as follows:
  \begin{itemize}
    \item An edge represented by a map $e:\Delta^1 \otimes \Delta^1 \to \bcat{D}$ is Cartesian if and only if it factors through $\Delta^1 \times \Delta^1$ and the restriction to $\Delta^{\set{1}}\times\Delta^1$ is an equivalence in $\bcat{D}$.
    \item A triangle represented by a map $\sigma:\Delta^1\otimes \Delta^2 \to \bcat{D}$ is lean if and only if its restriction to $\Delta^{\set{1}}\times \Delta^2$ is thin in $\bcat{D}$.
    \item A triangle represented by a map $\sigma:\Delta^1\otimes \Delta^2  \to \bcat{D}$  is thin if and only if it is lean and its restriction to $\Delta^{\{0\}}\times \Delta^2$ is thin.
  \end{itemize}
\end{proposition}

\begin{proof}
  Since, over an $\infty$-bicategory, the definition of 2-Cartesian we provided in \cite{AGS_CartI} coincides with the notion of \emph{outer 2-Cartesian fibration} from \cite{GHL_Cart}, this follows immediately from \cite[Thm 2.2.6]{GHL_Cart}.
\end{proof}

We can leverage this definition to give an extension of the Gray product, which more fully captures the decoration in this case.

\begin{definition}
  Let $X\in \mbsSet$ and denote by $\widetilde{X}$ its underlying scaled simplicial set. We define $\Delta^1\widehat{\otimes} X\in \mssSet$ extending $\Delta^1 \otimes \widetilde{X}$ by declaring
  \begin{itemize}
    \item A 1-simplex $(\sigma_1,\sigma_X)$ is marked if it is degenerate, or if $\sigma_1$ is degenerate on $\{1\}$, and $\sigma_X$ is marked in $X$. 
    \item A 2-simplex $(\sigma_1,\sigma_X)$ is thin if \emph{any} of the following conditions hold.
    \begin{itemize}
      \item The simplex $(\sigma_1,\sigma_X)$ is thin in $\Delta^1\otimes \widetilde{X}$. 
      \item The simplex $\sigma_X$ is lean in $X$ and $\sigma_1(1 \to 2)$ is degenerate on $1$.
      \item The simplex $\sigma_X$ is lean in $X$, $\sigma_X(0\to 1)$ is marked in $X$ and the simplex $\sigma_1$ is of the form $0 \to 0 \to 1$.
    \end{itemize}
  \end{itemize}  
  For $X,Y\in \mssSet$, we can then define $\Fun^{\widehat{\on{gr}}}(\Delta^1,Y)\in \mbsSet$ via the adjunction
  \[
  \Hom_{\mbsSet}(S,\Fun^{\widehat{\on{gr}}}(\Delta^1,Y))\cong \Hom_{\mssSet}(\Delta^1\otimes S,Y).
  \] 
\end{definition}

\begin{remark}
  For $X=(X,M_X,T_X\subset C_X)$, we will denote by $\{0\}\widehat{\otimes} X$ the full \textbf{MS} simplicial subset of $\Delta^1\widehat{\otimes} X$ corresponding to $\{0\}\times X$. This is isomorphic to the \textbf{MS} simplicial set $(X,\flat,T_X)$. Similarly, we denote by $\{1\}\widehat{\otimes} X$ the \textbf{MS} simplicial set $(X,M_X,C_X)$.
\end{remark}

We can then define the free 2-Cartesian fibration.

\begin{definition}
  Let $p:\bcat{X}\to \bcat{D}$ be a functor of $\infty$-bicategories. Denote by $\bcat{X}^\natural=(\bcat{X},M_X,T_X\subset T_X)$ the associated \textbf{MB} simplicial set in which the equivalences are marked. We define an \textbf{MB} simplicial set $\mathbb{F}(\bcat{X})^\natural$ as the pullback
  \[
  \begin{tikzcd}
  \mathbb{F}(\bcat{X})^\natural\arrow[d]\arrow[r] & \Fun^{\widehat{\on{gr}}}(\Delta^1,\bcat{D})\arrow[d,"\on{ev}_1"] \\
  \bcat{X}^\natural \arrow[r,"p"'] & \bcat{D} 
  \end{tikzcd}
  \]
  We denote the natural map induced by evaluation at $0$ by $\func{\mathbb{F}(p):\mathbb{F}(\bcat{X})^\natural \to \bcat{D}}$. 
\end{definition}

\begin{proposition}\label{prop:freefib2Cart}
  Let $p:\bcat{X}\to \bcat{D}$ be a functor of $\infty$-bicategories. Then 
  \[
  \func{\mathbb{F}(p):\mathbb{F}(\bcat{X})^\natural \to \bcat{D}}
  \]
  is a 2-Cartesian fibration
\end{proposition}

This proposition will follow from a somewhat technical lemma. 

\begin{lemma}\label{lem:FreeFibPP}
  Let $\func{f:A\to B}$ be an \bS-anodyne morphism. Then 
  \[
  \func{
    \Delta^1\widehat{\otimes} A \coprod_{\{0\}\widehat{\otimes} A} \{0\}\widehat{\otimes} B \to \Delta^1 \widehat{\otimes} B 
  }
  \] 
  is $\mathbf{MS}$-anodyne.
\end{lemma}

\begin{proof}
  As usual, we can check on generating \bS-anodyne morphisms. Before commencing the proof, we will make a preliminary definition. We say that a morphism of $\mathbf{MS}$ simplicial sets is of type $(\heartsuit)$ if it is in the weakly saturated hull of morphisms of the type described in \autoref{lem:Ui_MS-anodyne}. We can now proceed to perfom our case-by-case analysis.
  \begin{itemize}
    \item[(A2)] We can scale $\set{1}\widehat{\otimes} \Delta^4$ using a pushout of type \ref{MS:wonky4}. The remaining 2-simplices can be scaled using morphisms of type $(\heartsuit)$.
    \item[(A5)] Is a pushout of type \ref{MS:inner}, a pushout of type \ref{MS:Cartlifts}, and a pushout of type \ref{MS:nhorn}.
    \item[(S1)] Is an iterated pushout of type \ref{MS:Compose} and morphisms of type $(\heartsuit)$.
    \item[(S2)] Is an isomorphism. 
    \item[(S3)] Is a pushout along a morphism of type $(\heartsuit)$.
    \item[(S4)] Is a pushout along morphisms  of type \ref{MS:composedeg4} and $(\heartsuit)$.
    \item[(S5)] Is a pushout along morphisms of type  \ref{MS:composemarked5} and $(\heartsuit)$. 
  \end{itemize}
  The remaining three cases are the horn inclusions.
  \begin{itemize}
    \item[(A1)] Since no morphisms are marked on either side and the lean and thin scalings are identical, this is a consequence of \cite[Prop 4.1.9]{GHL_LaxLim}.
    \item[(A4)] Let us set the notation $(\Delta^n)^{\dagger}=(\Delta^n,\Delta^{\set{n-1,n}},\flat \subset \Delta^{\set{0,n-1,n}})$ and let us similarly define $(\Lambda^n_n)^{\dagger}$ . First we will define an order for the simplices of maximal dimension in $\Delta^1 \widehat{\otimes}(\Delta^n)^\dagger$. Let $\theta_\epsilon:\Delta^{n+1} \to \Delta^1 \widetilde{\otimes} (\Delta^n)^\dagger$ for $\epsilon \in \set{0,1}$ and let $\nu_{\theta_{\epsilon}}$ be the first element in $\Delta^{n+1}$ such that the value of $\theta_\epsilon$ at $\nu_{\theta_{\epsilon}}$ has the first coordinate equal to $1$. We say that $\theta_1 < \theta_2$ if and only if $\nu_{\theta_1}<\nu_{\theta_2}$. We further observe every simplex of maximal dimension is uniquely determined by the value $\nu_\theta$. Consequently we will denote by $\theta_i$ for $i\in \set{1,2,\dots,n+1}$ the unique simplex of maximal dimension that has $\nu_{\theta_i}=i$. We can produce now a filtration
    \[
      \Delta^1 \widehat{\otimes} (\Lambda^n_n)^\dagger \to  Y_{n+1} \to Y_{n} \to \cdots \to  Y_2 \to Y_1=\Delta^1 \widetilde{\otimes} (\Delta^n)^{\dagger}
    \]
    where $Y_j$ is the full $\mathbf{MS}$ subsimplicial set of $\Delta^1 \widetilde{\otimes} (\Delta^n)^{\dagger}$ containing the simplices of $Y_{j+1}$ in addition to the simplex $\theta_j$. It is an straightforward to see that the first map is a pushout along the inner-horn inclusion $\Lambda^{n+1}_n \to \Delta^{n+1}$. Since the edge $n-1 \to n$ is marked in $(\Delta^{n})^{\dagger}$ it follows that the triangle $\set{n-1,n,n+1}$ is thin in $\theta_{n+1}$. The rest of the morphisms are in the weakly satured hull of morphisms of type \ref{MS:nhorn}. Each map $Y_{j+1} \to Y_j$ is obtained precisely after taking two pushouts of type \ref{MS:nhorn}. First we had the face missing $j-1$ of $\theta_j$ and then we add the whole simplex missing.
    
    \item[(A3)] The argument here is precisely dual to the previous case, replacing `marked edge'' with ``degenerate edge'' and ``\ref{MS:nhorn}'' with ``\ref{MS:0horn}''. Note that we also need to reverse the order in which the add the simplices of maximal dimension that are missing.
  \end{itemize}
\end{proof}

\begin{proof}[Proof (of \ref{prop:freefib2Cart})]
  Given a lifting problem 
  \[
  \begin{tikzcd}
  A \arrow[r]\arrow[d,"f"'] & \FF(\bcat{X})^\natural\arrow[d] \\
  B \arrow[r] & \bcat{D}
  \end{tikzcd}
  \]
  where $f$ is \bS-anodyne, we need to find a diagram 
  \[
  \begin{tikzcd}
   \{1\}\widehat{\otimes} B \arrow[r] \arrow[d,hookrightarrow] & \bcat{X} \arrow[d,"p"]\\
    \Delta^1\widehat{\otimes}B  \arrow[r] & \bcat{D}
  \end{tikzcd}
  \]
  extending the diagram
  \[
  \begin{tikzcd}
  \{1\}\widehat{\otimes} A \arrow[r]  \arrow[d,hookrightarrow] & \bcat{X} \arrow[d,"p"']\\
   \Delta^1\widehat{\otimes} A \coprod\limits_{\{0\}\widehat{\otimes} A}\{0\}\widehat{\otimes} B  \arrow[r] & \bcat{D}
  \end{tikzcd}
  \]
  defined by the lifting problem. We first note that, since $\{1\}\widehat{\otimes}A\to \{1\}\widehat{\otimes}B$ is, in particular, \textbf{MS}-anodyne, we can solve the lifting problem on the bottom. It thus remains for us to solve the extension problem
  \[
  \begin{tikzcd}
  \{1\}\widehat{\otimes}B\coprod\limits_{\{1\}\widehat{\otimes}B}\Delta^1\widehat{\otimes} A \coprod\limits_{\{0\}\widehat{\otimes} A}\{0\}\widehat{\otimes} B\arrow[r]\arrow[d] & \bcat{D}^\natural\\
  \Delta^1\widehat{\otimes} B
  \end{tikzcd}
  \]
  However, the morphism on the left fits into a composite. 
  \[
  \begin{tikzcd}
  \Delta^1\widehat{\otimes} A \coprod\limits_{\{0\}\widehat{\otimes} A}\{0\}\widehat{\otimes} B\arrow[r] & \{1\}\widehat{\otimes}B\coprod\limits_{\{1\}\widehat{\otimes}B}\Delta^1\widehat{\otimes} A \coprod\limits_{\{0\}\widehat{\otimes} A}\{0\}\widehat{\otimes} B\arrow[r] & \Delta^1\widehat{\otimes} B
  \end{tikzcd}
  \]
  where the first morphism is a pushout by an \emph{MS}-anodyne morphisms, and the composite is one of the morphisms from \autoref{lem:FreeFibPP}. Consequently, the morphism on the left is an \textbf{MS} trivial cofibration by 2-out-of-3, and thus the lifting problem has a solution.
\end{proof}

\begin{definition}
  Let $p:\bcat{X}\to \bcat{D}$ be a functor of $\infty$-bicategories. We denote by $\bcat{X}_{d\upslash}$ the fibre of $\func{\mathbb{F}(p):\mathbb{F}(\bcat{X})^\natural \to \bcat{D}}$ over $d\in \bcat{D}$. Note that \autoref{prop:freefib2Cart} implies that $\bcat{X}_{d\upslash}$ is an $\infty$-bicategory.
\end{definition}

\begin{remark}
  Unwinding the definition of $\bcat{X}_{d\upslash}$, we see that a morphism from $f:d\to p(x)$ to $g:d\to p(y)$ is given by a diagram
  
  \[
  \begin{tikzcd}
  d \arrow[d,equal]\arrow[dr,""{name=U}]\arrow[r,"f"] & |[alias=UR]| p(x)\arrow[d,"p(u)"]\\
  |[alias=DR]| d \arrow[r,"g"'] & p(y) \arrow[from=U,to=UR,equals,shorten <=0.6em,shorten >= 0.6em] \arrow[from=U, to=DR,Rightarrow,shorten <= 0.6em]
  \end{tikzcd}
  \]
  which, in a strict 2-category, we could view as a diagram 
  
  \[
  \begin{tikzcd}
  & |[alias=Z]| p(x)\arrow[dd,"p(u)"]\\
  d\arrow[ur,"f"]\arrow[dr,"g"',""{name=U}] & \\
  & p(y) \arrow[from=Z,to=U,Rightarrow,shorten <=1.5em,shorten >=1em]
  \end{tikzcd}
  \]
  which justifies our notation. We will later see that, for a functor $F:\CC\to \DD$ of strict 2-categories, there is an equivalence of $\infty$-bicategories
  \[
  \Nerv^{\sc}(\CC_{d\upslash})\simeq \Nerv^{\sc}(\CC)_{d\upslash}.
  \]
  connecting our free fibration to more familiar notions. 
\end{remark}

\begin{remark}
  Let $p:\bcat{X}\to \bcat{D}$ be a functor of $\infty$-bicategories. There is a cofibration 
  \[
  \func{
    \gamma_{\bcat{X}}: \bcat{X} \to \mathbb{F}(\bcat{X})^\natural 
  }
  \]
  over $\bcat{D}$, which sends a simplex $\Delta^n\to \bcat{X}$ to the map $\Delta^1 \widehat{\otimes} \Delta^n\to \Delta^0\otimes \Delta^n \to \bcat{X}\to \bcat{D}$.
\end{remark}

To simplify our examination of this map, we provide a way of constructing `almost degenerate' $(n+1)$-simplices in $\mathbb{F}(\bcat{X})^\natural$ from $n$-simplices in $\mathbb{F}(\bcat{X})^\natural$. 

\begin{construction}\label{constr:extensions}
  For every $0\leq j\leq n$, we define a map 
  \[
  \func*{
    E_j:\Delta^1\times \Delta^{n+1}\to \Delta^1\times \Delta^n;
    (m,r) \mapsto \begin{cases}
    (m,r) & r\leq j\\
    (1,r-1) & r>j 
    \end{cases}
  }
  \]
  It is easy to check that this map respects scalings, and thus yields $E_j:\Delta^1\otimes \Delta^{n+1}\to \Delta^1\otimes \Delta^n$. Moreover, the induced map $\{1\}\times \Delta^{n+1}\to \{1\}\times \Delta^n$ is precisely the degeneracy map $s_j$. 
  
  Given an $n$-simplex $\sigma:\Delta^n\to \mathbb{F}(\bcat{X})$ (possibly having some non-trivial decorations) defined by $\phi^\sigma:\Delta^1 \widehat{\otimes} \Delta^n\to \bcat{D}$ and $\rho^\sigma:\{1\}\times \Delta^n\to \bcat{X}$, we define an $(n+1)$-simplex $E_j^\ast(\sigma)$ by
  \[
  \begin{tikzcd}
  \Delta^1\otimes \Delta^{n+1}\arrow[r,"E_j"] & \Delta^1 \widehat{\otimes} \Delta^n \arrow[r,"\phi^\sigma"] & \bcat{D}\\
  \Delta^{n+1}\times \{1\}\arrow[r,"s_j"']\arrow[u,hookrightarrow] & \{1\}\times\Delta^n\arrow[r,"\rho^\sigma"']\arrow[u,hookrightarrow] & \bcat{X}\arrow[u,"p"']
  \end{tikzcd}
  \]
  We will call $E^\ast_j(\sigma)$ the \emph{$j$\textsuperscript{th} extension of $\sigma$}.
  
  Given a simplex $\sigma:\Delta^n\to \mathbb{F}(\bcat{X})$ as above, we will denote by $\ell^\sigma_0$ the corresponding map $\{1\}\times \Delta^n\to \bcat{X}$. 
\end{construction}

The following lemmata follow immediately from the definition. 

\begin{lemma}\label{lem:facesofextensions}
  Let $\sigma:\Delta^n\to \mathbb{F}(\bcat{X})$. Then the faces of $E_j^\ast(\sigma)$ can be written as follows. 
  \begin{itemize}
    \item If $j=n$, and $s=n+1$, then $d_s(E_j^\ast(\sigma))=\sigma$. 
    \item If $j+1<s\leq n+1$, then $d_s(E_j^\ast(\sigma)=E_j^\ast(d_{s-1}(\sigma))$. 
    \item If $0\leq s<j$, then $d_s(E^\ast_j(\sigma))=E_{j-1}(d_s(\sigma))$. 
    \item If $s=j+1$, then  $d_{j+1}(E^\ast_j(\sigma))=d_{j+1}(E_{j+1}^\ast(\sigma))$.
    \item If $s=j$ and $s\neq 0$, then $d_j(E^\ast_{j}(\sigma))=d_j(E^\ast_{j-1}(\sigma))$
    \item If $s=j=0$, then $d_0(E_0^\ast(\sigma))=\gamma_{\bcat{X}}(d_0(\ell_0^\sigma))$. 
  \end{itemize}
\end{lemma}

\begin{lemma}\label{lem:degofextensions}
  If $\sigma:\Delta^n\to \mathbb{F}(\bcat{X})$ is degenerate, then for every $0\leq j\leq n$, the $j$\textsuperscript{th} extension $E_j^\ast(\sigma)$ is degenerate. 
\end{lemma}

\begin{lemma}\label{lem:extensionofextension}
  Let $\sigma:\Delta^{n-1}\to \mathbb{F}(\bcat{X})$. Then for every $0\leq j\leq n-1$ and every $0\leq i\leq n$, the simplex $E_i^\ast(E_j^\ast(\sigma))$ is degenerate.
\end{lemma}

\begin{theorem}\label{thm:thebigfreeboi}
  Let $\func{p:\bcat{X}\to \bcat{D}}$ be a functor of $\infty$-bicategories.  Then the morphism 
  \[
  \func{\gamma_{\bcat{X}}: \bcat{X}^\natural \to \FF(\bcat{X})^{\natural}}
  \]
  is \bS-anodyne over $\bcat{D}$.
\end{theorem}

\begin{proof}
  Let us start by defining $Z_0$ to be the subsimplicial set of $\FF(\bcat{X})$ consisting of all of the simplices belonging to $\bcat{X}$, all of the $0$-simplices of $\FF(\bcat{X})$ and all of the possible $j$-extensions of the $0$-simplices. We extend this definition inductively by defining $Z_n$ to consist of all of the simplices of $Z_{n-1}$, all of the $n$-simplices of $\FF(\bcat{X})$ and the $(n+1)$-simplices appearing as extensions of $n$-simplices. 
  We set the convention $Z_{-1}=X$ and we fix the notation
  \[
  \func{\gamma_{n-1}:Z_{n-1} \to Z_{n}.}
  \]
  Observe that since \bS-anodyne maps are stable under transfinite composition it will suffice to show that each $\gamma_{n-1}$ is \bS-anodyne.
  
  We start by analyzing $\gamma_{-1}$. Observe that given an object $\sigma:\Delta^0 \to \FF(\bcat{X})$ we can consider the Cartesian edge $E_0^\ast(\sigma):\Delta^1 \to \FF(\bcat{X})$. Since the target of $E_0^\ast(\sigma)$ is already contained in $\bcat{X}$ it follows that we can add $\sigma$ by means of a pushout along a \bS-anodyne map. Repeating this process for each object in $\FF(\bcat{X})$ conclude that  $\gamma_{-1}$ is \bS-anodyne.

  Now we will tackle the general case for $\gamma_{n-1}$ with $n\geq 1$. Let us pick an order on the set of non-degenerate $n$-simplices of $\FF(\bcat{X})$ that are not already contained in $Z_{n-1}$. For every $\sigma \in \FF(\bcat{X})_{n}$ we define $Z_{n-1}(\sigma)$ as the subsimplicial subset of $Z_{n}$ containing all of the simplices of $Z_{n-1}$ in addition to the $n$-simplices $\theta \leq \sigma$ and its corresponding extensions. Let $\on{suc}(\sigma)$ be the successor of $\sigma$ in our chosen order. We will adopt the convention $Z_{n-1}(\emptyset)=Z_{n-1}$ and $\on{suc}(\emptyset)=\sigma_0$ is the first element in our ordering. To show that $\gamma_{n-1}$ is \bS-anodyne it will suffice to prove that 
  \[
  \func{Z_{n-1}(\sigma) \to Z_{n-1}(\rho)}
  \]
  is \bS-anodyne where $\rho=\on{suc}(\sigma)$.
  
  The proof will be divided into three cases. First let us assume that for every $1\leq j \leq n$ all of the faces of $E_j^\ast(\rho)$ are contained in $Z_{n-1}(\sigma)$ except the faces missing $j+1,j$. Applying \autoref{lem:facesofextensions} for $j=0$ yields 
  \[
  d_s E_0^\ast(\rho)=\begin{cases}
  E_0^\ast\left( d_s(\rho) \right), \enspace \text{ if } 1 < s \leq n+1 \\
  d_{1}\left(E_1^\ast(\rho) \right), \enspace \text{ if } s=1\\
  \gamma_X(d_0(\ell_0)), \enspace \text{ if } s=0.
  \end{cases}
  \]
  which shows that all of the faces of $E_0^\ast(\rho)$ are already in $Z_{n-1}(\sigma)$ except the $1$-face. By construction the triangle $\Delta^{\set{0,1,2}}$ is thin in $E_0^\ast(\rho)$, which shows that we can add the simplex $E_0^\ast(\rho)$ via a pushout along a \bS-anodyne map. Let us denote by $V_0$ the resulting simplicial set $Z_{n-1}(\sigma) \to V_0 \to Z_{n-1}(\rho)$. Using \autoref{lem:facesofextensions} again, we see that all of the faces of $E_1^\ast(\rho)$ are in $V_0$ except the $2$-face. A similar argument as above shows that we can add $E_1^\ast(\rho)$ in a \bS-anodyne way and thus obtaining a new subsimplicial set that we denote $V_1$. We can repeat this process until we reach $V_{n-1}$. In our final step we observe that we have a pullback diagram
  \[
  \begin{tikzcd}[ampersand replacement=\&]
  \Lambda^{n+1}_{n+1} \arrow[r] \arrow[d] \& \Delta^{n+1} \arrow[d,"E_n^\ast(\rho)"] \\
  V_{n-1} \arrow[r] \& Z_{n-1}(\rho)
  \end{tikzcd}
  \]
  where the last edge of $\Lambda^{n+1}_{n+1}$ is 2-Cartesian in $\FF(\bcat{X})$ and the triangle $\Delta^{\set{0,n,n+1}}$ is coCartesian. Therefore we can add $E_n^\ast(\rho)$ using a pushout along a \bS-anodyne map and conclude that $Z_{n-1}(\sigma) \to Z_{n-1}(\rho)$ is in this case \bS-anodyne.
  
  For the second case let us suppose that there exists some $1 \leq \alpha \leq n$ such that $d_{\alpha}\left( E_{\alpha}^\ast(\rho)\right)$ is already in $Z_{n-1}(\sigma)$ and that for every $j>\alpha$ we have that the faces missing $j+1$, $j$ in $E_j^\ast(\rho)$ are not contained in $Z_{n-1}(\sigma)$. We claim that for every $0\leq k < \alpha$ the simplex $E_k^\ast(\rho) \in Z_{n-1}(\rho)$. One can easily check that
  \[
  E_k^\ast(\rho)= E_k^\ast(d_{\alpha}\left(E_{\alpha}^\ast(\rho)\right)), \enspace 0\leq k < \alpha.
  \]
  In particular, this shows that it will suffice to show that all of the extensions of $d_{\alpha}\left( E_{\alpha}^\ast(\rho)\right)$ are contained in $Z_{n-1}(\sigma)$. To provide a proof of this latter claim we observe that $d_{\alpha}\left( E_{\alpha}^\ast(\rho)\right)\in Z_{n-1}(\sigma)$ if and only at least one of the following conditions hold:
  \begin{itemize}
    \item[*)] The face  $d_{\alpha}\left( E_{\alpha}^\ast(\rho)\right)$ is contained in $\bcat{X}$.
    \item[i)] The face $d_{\alpha}\left( E_{\alpha}^\ast(\rho)\right)$ is degenerate.
    \item[ii)] The face $d_{\alpha}\left( E_{\alpha}^\ast(\rho)\right)$ is the extension of an $(n-1)$-simplex.
    \item[iii)] There exists $\theta \leq \sigma$,  such that $d_{\alpha}\left( E_{\alpha}^\ast(\rho)\right)$ is a face of an extension of $\theta$
  \end{itemize}
  If condition *) holds then it is easy to see that all of the possible extensions are already in $Z_{n-1}(\sigma)$. Using \autoref{lem:degofextensions} and \autoref{lem:extensionofextension} we see that the claim holds if the conditions i) or ii) are satisfied. Suppose now that condition iii) holds. We can assume without loss of generality that $d_{\alpha}(E_{\alpha}^\ast(\rho))=d_{\beta}(E_{\beta}^\ast(\theta))$. A straightforward computation shows that
  \[
  E_j^\ast(d_{\beta}(E_\beta^\ast(\theta)))=\begin{cases}
  s_j(d_{\beta}(E_\beta^\ast(\theta))), \enspace \text{ if }j \geq \beta \\
  E_j^\ast(\theta), \enspace \text{ if } j < \beta
  \end{cases} 
  \]
  so again, the claim holds. We have shown $Z_{n-1}(\sigma)=V_{\alpha-1}$ and thus the previous argument runs exactly the same way.
  
  The last case to analyze is the degenerate situation where $\rho$ is already in $Z_{n-1}(\sigma)$. In this case we need to show that $Z_{n-1}(\sigma)=Z_{n-1}(\rho)$, i.e. we need to show that we already have all of the extensions of $\rho$. Since $\rho \notin Z_{n-1}$ and it is not degenerate it follows that $\rho=d_{\beta}(E_{\beta}(\theta))$ for some $\theta \leq \sigma$. Using the same reasoning as before we can see that $E_k(\rho) \in Z_{n-1}(\sigma)$ for all $0\leq k \leq n$ and the claim follows.

  In order to finish the proof there is one last thing we have to take care of in the filtration, namely, the decorations. We need to show that whenever we add a marked edge (resp. lean, resp. thin triangle) in our filtration we can add the decoration to our filtration in a \bS-anodyne way. For the marked edges this essentially an specific case of the proof given in \autoref{cor:decormarked}. We leave the rest of the decorations as an exercise for the reader.
\end{proof}

\begin{remark}
  The morphism $\gamma_{\bcat{X}}: \bcat{X}^\natural \to \mathbb{F}(\bcat{X})^{\natural}$ we can viewed as the unit of a bicategorical free-forgetful adjunction between the $\infty$-bicategory of of $\infty$-bicategories over $\bcat{D}$ and the $\infty$-bicategory of 2-Cartesian fibrations over $\bcat{D}$. We will not pursue this direction further in this document. A detailed study of bicategorical adjunctions is part of the research program of the authors and will appear in future work.
\end{remark}

\begin{definition}\label{def:freemarking}
  Let $\func{p:\bcat{X} \to \bcat{D}}$ be a functor of $\infty$-bicategories. Assume that $\bcat{X}$ comes equipped with a marking containing all of the equivalences and denote the resulting marked $\infty$-bicategory by $\bcat{X}^{\dagger}$.  We define new marking on $\FF(\bcat{X})$  as follows. We declare and edge represented by $\Delta^1 \otimes \Delta^1 \to \bcat{D}$ marked if and only if it factors through $\Delta^1 \times \Delta^1$ and its restriction to $\Delta^{\set{1}}\times \Delta^{1}$ factors through a marked edge in $\bcat{X}$. We define marked-scaled simplicial $\FF(\bcat{X})^{\dagger}$ having the same lean and thin triangles as $\mathbb{F}(\bcat{X})^{\natural}$ but equipped with this new collection of marked edges.
\end{definition}

\begin{corollary}\label{cor:decormarked}
  Let $\func{p:\bcat{X} \to \bcat{D}}$ be a functor of $\infty$-bicategories. Assume that $X$ comes equipped with a marking (containing the equivalences) and denote the corresponding marked $\infty$-bicategory by $\bcat{X}^{\dagger}$. Then the morphism
  \[
  \func{ \bcat{X}^{\dagger} \to \FF(\bcat{X})^{\dagger}}
  \]
  is \bS-anodyne.
\end{corollary}
\begin{proof}
  Let us consider the pushout diagram
  \[
  \begin{tikzcd}
  \bcat{X} \arrow[r] \arrow[d] & \bcat{X}^{\dagger} \arrow[d] \\
  \FF(\bcat{X})^\natural \arrow[r] & \FF(\bcat{X})^{\diamond}
  \end{tikzcd}
  \]
  where it follows from \autoref{thm:thebigfreeboi} that the left-most vertical map is \bS-anodyne. To finish the proof we just need to show that the morphism $\FF(\bcat{X})^{\diamond} \to \FF(\bcat{X})^{\dagger}$ is again anodyne.  Let $e: \Delta^1 \otimes \Delta^1 \to \bcat{D}$ be a marked edge of $\FF(\bcat{X})^{\dagger}$. First we observe that $E_0^\ast(e)$ is a thin 2-simplex such that $d_0(E_0^\ast(e))$ and $d_2(E_0^\ast(e))$ are marked in $\FF(\bcat{X})^{\diamond}$. It particular it follows that we can marked the edge $d_1(E_0^\ast(e))$ using a pushout along a morphism of type \ref{mb:composeacrossthin}.  Using a pushout along a morphism of the type described in \cite[Lem. 3.7]{AGS_CartI} we can mark all edges in $E_1^\ast(e)$. We conclude the proof after noting that $d_2(E_1^\ast(e))=e$.
\end{proof}

\subsection{Partially lax colimits and cofinality}\label{subsec:cofcrit}

We now turn to our main result of this section, a criterion for higher cofinality. We will not here recapitulate the theory of higher (co)limits expounded in \cite{GHL_LaxLim}, but see \autoref{rmk:GHL_colim_recap} for details on the connection with $(\infty,2)$-categorical colimits. 

\begin{definition}\label{defn:markedcofinal}
  Let $X^{\dagger},Y^{\dagger}$ be a pair of marked-scaled simplicial sets and consider a marking preserving functor, $f:X^\dagger \to Y^\dagger$. We say that $f$ is a marked cofinal functor (or simply cofinal) if the associated functor of marked-biscaled simplicial sets 
  \[
  \func{f:(X,E_X,T_X \subset \sharp) \to (Y,E_Y,T_Y \subset \sharp) }
  \]
  is a weak equivalence in model structure of $\bS$ simplicial sets over $Y$.
\end{definition}

\begin{remark}\label{rmk:GHL_colim_recap}
  The theory of partially lax colimit in $\infty$-bicategories was independently developed by Berman in \cite{Berman}, the present author in \cite{AGSRelNerve} and \cite{AG_cof}, and Gagna, Harpaz, and Lanari in \cite{GHL_LaxLim}. The latter provides a full characterization of marked (co)limits in $\infty$-bicategories, including the four variances which arise from changing the directions of 2-morphisms in the corresponding notion of cone. The theory of partially lax colimit described in \cite{AG_cof} corresponds to the case of \emph{outer colimits} in the language of \cite{GHL_LaxLim}. 
  
  By \cite[Thm 5.4.4]{GHL_LaxLim}, a functor of marked $\infty$-bicategories $f:\bcat{C}^\dagger\to \bcat{D}^\dagger$ is \emph{outer cofinal} --- i.e., pullback along $f$ preserves outer colimits --- if and only if $f$ is marked cofinal in the sense of \autoref{defn:markedcofinal} above (compare \cite[Defn 4.3.3]{GHL_LaxLim} to \cite[Defn 3.25]{AGS_CartI} and  \cite[Prop. 3.28]{AGS_CartI} to see that these conditions do indeed coincide).
\end{remark}

\begin{remark}
  Let $f:(\bcat{C},E_{\bcat{C}},T_{\bcat{C}}) \to (\bcat{D},E_{\bcat{D}},T_{\bcat{D}})$ be a functor of marked $\infty$-bicategories. Observe that in order to see if $f$ is cofinal we can assume that the markings of both $\infty$-bicategories contain all equivalences. Indeed, this follows easily after taking pushouts along morphisms of type \ref{mb:equivalences}. Consequently for the rest of the section we will assume that the markings satisfy this property.
\end{remark}

\begin{lemma}\label{lem:IffII}
  Let $f:\bcat{C}^{\dagger} \to \bcat{D}^{\dagger}$ be a functor of marked $\infty$-bicategories and recall from \autoref{def:freemarking} the associated marking on $\mathbb{F}(\bcat{C})^{\dagger}=(\mathbb{F}(\bcat{C}),E_{\mathbb{F}(\bcat{C})^\dagger},T_{\mathbb{F}(\bcat{C})^\dagger})$. Then the induced morphism
  \[
  \func{(\bcat{C},E_{\bcat{C}^\dagger},T_{\bcat{C}^\dagger}\subset \sharp) \to (\mathbb{F}(\bcat{C}),E_{\mathbb{F}(\bcat{C})^\dagger},T_{\mathbb{F}(\bcat{C})^\dagger}\subset \sharp)}
  \] 
  is \bS-anodyne.
\end{lemma}

\begin{proof}
  Let us denote $\bcat{C}^{\dagger}_{\sharp}=(\bcat{C},E_{\bcat{C}^\dagger},T_{\bcat{C}^\dagger}\subseteq \sharp)$ and similarly $\mathbb{F}(\bcat{C})^\dagger_{\sharp}$. We consider a pushout diagram over $\bcat{D}$ 
  \[
  \begin{tikzcd}[ampersand replacement=\&]
  \bcat{C}^{\dagger} \arrow[d,"\simeq"] \arrow[r] \& \bcat{C}^{\dagger}_{\sharp} \arrow[d,"\simeq"] \\
  \mathbb{F}(\bcat{C})^{\dagger} \arrow[r] \& \mathbb{F}(\bcat{C})^{\dagger}_{\diamond} 
  \end{tikzcd}
  \]
  whose vertical morphisms are all weak equivalences. We will show that the induced map $s:\mathbb{F}(\bcat{C})^{\dagger}_{\diamond}  \to \mathbb{F}(\bcat{C})^{\dagger}_{\sharp}$ is anodyne. Let $\sigma: \Delta^2 \to \mathbb{F}(\bcat{C})^{\dagger}_{\diamond}$ be a triangle. Note that by construction $E_0^\ast(\sigma)$ is fully lean scaled. This implies that $E_0^\ast(\sigma)$ has all faces lean except possibly the face missing the vertex $1$. Since the triangle $\set{0,1,2}$ is thin we can take a pushout along an \bS-anodyne morphism to lean scaled the face missing $1$. 
  
  Now we consider $E_1^\ast(\sigma)$ and observe that the face missing $0$ and $3$ are lean. Additionally we see that $d_1(E_1^\ast(\sigma))=d_1(E_0^\ast(\sigma))$ so it follows that all faces are already lean except the face missing the vertex $2$. We scale the aforementioned face after noting that $\set{1,2,3}$ is a thin triangle. A similar argument then shows that all of the faces of $E_2^\ast(\sigma)$ are scaled except possible $d_3(E_2^\ast(\sigma))=\sigma$. However since the last vertex is marked and the triangle $\set{0,2,3}$ is scaled the result follows.
\end{proof}

\begin{definition}
  Let $f:\bcat{C}^{\dagger} \to \bcat{D}^{\dagger}$ be a functor of marked $\infty$-bicategories. Given $d \in \bcat{D}$ we denote by $\bcat{C}^{\dagger}_{d\upslash}$ the fibre over the object $d$ of the morphism $\mathbb{F}(\bcat{C})^\dagger \to \bcat{D}$.
\end{definition}

\begin{definition}\label{def:upcomafibrant}
  Let $f:\bcat{C}^{\dagger} \to \bcat{D}^{\dagger}$ be a functor of marked $\infty$-bicategories. We define the 2-Cartesian fibration $\bcat{C}^\dagger_{\bcat{D}\upslash} \to \bcat{D}$ to be a fibrant replacement of the object $(\mathbb{F}(\bcat{C}),E_{\mathbb{F}(\bcat{C})^\dagger},T_{\mathbb{F}(\bcat{C})^\dagger}\subset \sharp)$ in $(\mbsSet)_{/\bcat{D}}$. 
\end{definition}

\begin{proposition}\label{prop:prebigtheorem}
  Let $f:\bcat{C}^{\dagger} \to \bcat{D}^\dagger$ be a functor of marked $\infty$-bicategories. Then the following statements are equivalent:
  \begin{itemize}
    \item[i)] The map $f$ is marked cofinal.
    \item[ii)] The morphism $(\mathbb{F}(\bcat{C}),E_{\mathbb{F}(\bcat{C})^\dagger},T_{\mathbb{F}(\bcat{C})^\dagger}\subset \sharp) \to (\mathbb{F}(\bcat{D}),E_{\mathbb{F}(\bcat{D})^\dagger},T_{\mathbb{F}(\bcat{D})^\dagger}\subset \sharp)$ is a weak equivalence.
    \item[iii)] For every $d \in \bcat{D}$ we have an equivalence of $\infty$-categorical localizations $L_W(\bcat{C}_{d\upslash}^\dagger) \to L_W(\bcat{D}_{d\upslash}^\dagger)$.
  \end{itemize}
\end{proposition}
\begin{proof}
  The equivalence $i) \iff ii)$ follows from \autoref{lem:IffII} and the functoriality of the free fibration. To finish the proof we will show that $ii) \iff iii)$.
  
  Consider projective-fibrant functors $\mathcal{F}_{\bcat{C}},\mathcal{F}_{\bcat{D}}:\mathfrak{C}^{\on{sc}}[\bcat{D}]^\op \to \on{Set}^{\mathbf{ms}}_\Delta$ equipped with equivalences $\on{Un}_{\bcat{D}}(\mathcal{F}_{\mathbb{C}})\simeq \mathbb{F}(\bcat{C})$ and $\on{Un}_{\bcat{D}}(\mathcal{F}_{\bcat{D}})\simeq \mathbb{F}(\bcat{D})$. We can define new functors $\mathcal{F}_{\bcat{C}}^\dagger$ and $\mathcal{F}_{\bcat{D}}^\dagger$ and a morphism $\mathcal{F}_{\bcat{C}}^\dagger\to \mathcal{F}_{\bcat{D}}^\dagger$ via pushout, e.g., 
  \[
  \begin{tikzcd}
  \ST_{\bcat{D}}(\FF(\bcat{C})) \arrow[r,hookrightarrow,"\sim"] \arrow[d] & \mathcal{F}_{\bcat{C}}\arrow[d] \\
  \ST_{\bcat{D}}(\FF(\bcat{C})^\dagger) \arrow[r,hookrightarrow,"\sim"]& \mathcal{F}_{\bcat{C}}^\dagger 
  \end{tikzcd}
  \]
  We thus see that the induced map on fibrant repaclements $\mathsf{R}(\UN_{\bcat{D}}(\mathcal{F}^\dagger_{\bcat{C}}))\to \mathsf{R}(\UN_{\bcat{D}}(\mathcal{F}^\dagger_{\bcat{C}}))$ is a model for $\bcat{C}^\dagger_{\bcat{D}\upslash}\to \bcat{D}^\dagger_{\bcat{D}\upslash}$. Moreover, the pushout is computed pointwise. Unraveling the definition, we note that for each $d\in \bcat{D}$ the square 
  \[
  \begin{tikzcd}
  \ST_\ast(\bcat{C}_{d\upslash}) \arrow[r,"\sim"]\arrow[d] & \ST_{\bcat{D}}(\FF(\bcat{C}))(d)\arrow[d]\\
  \ST_\ast(\bcat{C}_{d\upslash}^\dagger) \arrow[r,"\sim"] & \ST_{\bcat{D}}(\FF(\bcat{C})^\dagger)(d)
  \end{tikzcd}
  \]
  is a homotopy pushout. We thus have natural equivalences 
  \[
  \bcat{C}_{d\upslash}^\dagger \simeq \ST_\ast(\CC_{d\upslash}^\dagger)\simeq \ST_{\bcat{D}}(\FF(\bcat{C})^\dagger)(d)\simeq \mathcal{F}_{\bcat{C}}^\dagger(d).
  \]
  Finally, we note that there are canonical natural identifications 
  \[
  \bcat{C}^\dagger_{\bcat{D}\upslash}\times_{\bcat{D}}\{d\}\simeq \mathsf{R}(\mathcal{F}_{\bcat{C}}^\dagger)(d)\simeq L_W(\mathcal{F}^\dagger_{\bcat{C}}(d))
  \]
  so that we get a commutative diagram 
  \[
  \begin{tikzcd}
  \bcat{C}^\dagger_{\bcat{D}\upslash}\times_{\bcat{D}}\{d\} \arrow[r,"\simeq"] \arrow[d]& L_W(\bcat{C}_{d\upslash}^\dagger)\arrow[d]\\
  \bcat{D}^\dagger_{\bcat{D}\upslash}\times_{\bcat{D}}\{d\} \arrow[r,"\simeq"] & L_W(\bcat{D}_{d\upslash}^\dagger)
  \end{tikzcd}
  \]
  The proposition then follows from \cite[Prop. 4.25]{AGS_CartI}.
\end{proof}

\begin{proposition}\label{prop:initial2version}
  Let $p:\bcat{X} \to \bcat{D}$ be a 2-Cartesian fibration such that every triangle in $\bcat{X}$ is lean. Suppose that for every $d \in \bcat{D}$ there exists an initial object $i_d \in \bcat{X}_d$ in the fibre over $d$. Then the restriction of $p$ to the the marked biscaled simplicial set spanned by initial objects  $\hat{p}: \hat{\bcat{X}} \to \bcat{D}$ is a trivial fibration of scaled simplicial sets.
\end{proposition}

\begin{proof}
  We first show that $\hat{p}$ is a fibration in the model structure on $\scsSet$. Since $p$ is a 2-Cartesian fibration, it is easy to see that $\hat{p}$ has the right lifting property against all scaled anodyne morphisms. By virtue of \cite[Cor 6.4]{GHL_Equivalence} it will suffice to check that $\hat{p}$ is an isofibration. Let $d_0 \to d_1=p(x_1)$ be an equivalence in $\bcat{D}$ and pick a lift $x_0 \to x_1$ such that $x_1$ is initial in the fibre over $d_1$. Let us pick an initial object $\hat{x}_0$  and consider the composite morphism $u:\hat{x}_0 \to x_1$. We claim that $u$ is an equivalence. Let $\scr{D} \subset \bcat{D}$ denote the underlying $\infty$-category of $\bcat{D}$ and let us consider a pullback diagram
  \[
  \begin{tikzcd}[ampersand replacement=\&]
  \bcat{X}_{\scr{D}} \arrow[d] \arrow[r] \& \bcat{X} \arrow[d,"p"] \\
  \scr{D} \arrow[r] \& \bcat{D}
  \end{tikzcd}
  \]
  where the left-most vertical morphism is a Cartesian fibration of $\infty$-categories. Let $\hat{\bcat{X}}_{\scr{D}}$ denote the restriction to the full subcategory on fibrewise initial objects. Then it follows from \cite[Prop. 2.4.4.9]{HTT} that the restriction $\hat{\bcat{X}}_{\scr{D}}\to \scr{D}$ is a trivial Kan fibration. In particular it detects equivalences and the claim follows.
  
  We have thus reduced our problem to showing that $\hat{p}$ is a bicategorical equivalence. By our hypothesis it follows that $\hat{p}$ is surjective on objects. To finish the proof we will check that for every pair of objects $x,y \in \bcat{X}$ the induced morphism of mapping $\infty$-categories
  \[
  \func{\hat{p}_{x,y}:\on{Map}_{\hat{\bcat{X}}}(x,y) \to \on{Map}_{\bcat{D}}(\hat{p}(x),\hat{p}(y))}
  \]
  is an equivalence. Note that since every 2-simplex in $\bcat{X}$ is lean it follows that not only is $p_{x,y}$ a coCartesian fibration, it is also a left fibration. Therefore we reduce our problem to showing that the fibres of $\hat{p}_{x,y}$ are all contractible. This follows from our hypothesis using \cite[Proposition 4.21]{AGS_CartI}
\end{proof}

\begin{lemma}
  Let $\bcat{D}^{\dagger}$ be a marked $\infty$-bicategory. Then the 2-Cartesian fibration $\bcat{D}^{\dagger}_{\bcat{D}\upslash} \to \bcat{D}$ satisfies the hypothesis of \autoref{prop:initial2version}.
\end{lemma}
\begin{proof}
  Recall the model for $\bcat{D}^{\dagger}_{\mathbb{D}\upslash}$ given in the proof of \autoref{prop:prebigtheorem}. As a direct consequence we observe that every triangle in $\bcat{D}^{\dagger}_{\mathbb{D}\upslash}$ is lean. We claim that for every $d \in \mathbb{D}$ the identity morphism $\on{id}_d$ on $d$ is initial in its corresponding fibre. Note that we can identify the fibre over $d$ with $L_W(\mathbb{D}^\dagger_{d\upslash})$. Since for every object $f:d \to d'$ , the mapping $\infty$-category $\on{Map}_{\mathbb{D}_{d\upslash}}(\on{id}_d,f)$ is contractible due to \autoref{lem:etaterminal} it follows that $\on{id}_d$ is initial in the localisation.
\end{proof}

\begin{lemma}\label{lem:etaterminal}
  Let $\bcat{D}$ be an $\infty$-bicategory. Let $\on{id}_d: d \to d$ and $e:d \to d'$ be a pair of edges in $\bcat{D}$ such that $\on{id}_d$ is degenerate. Let $r:\Delta^1 \times \Delta^1 \to \Delta^1$ be the morphism that sends every vertex to $0$ except $(1,1)$ which gets sent to $1$. Then the composite
  \[
  \func{\eta_e:\Delta^1 \times \Delta^1 \to[r] \Delta^1 \to[e] \bcat{D}}
  \] 
  defines a terminal object in the mapping $\infty$-category $\on{Map}_{\bcat{D}_{d\upslash}}(\on{id}_d,f)$.
\end{lemma}
\begin{proof}
  We will show that every boundary $\partial \alpha: \partial \Delta^n \to \on{Map}_{\bcat{D}_{d\upslash}}(\on{id}_d,f)$ such that $\partial \alpha  (n)=\eta_e$ can be extended to an $n$-simplex $\alpha:\Delta^n \to \on{Map}_{\bcat{D}_{d\upslash}}(\on{id}_d,f)$.
  
  We define a subsimplicial subset (with the inherited scaling) $\mathcal{S}^{n+1} \subset \Delta^1\otimes \Delta^{n+1}$  consisting of precisely those simplices $\sigma$ satisfying at least one of the conditions below:
  \begin{itemize}
    \item The simplex $\sigma$ is contained in $\Delta^{\set{0}}\times \Delta^{n+1} $.
    \item Given $j \in [n+1]$ the simplex $\sigma$ skips vertices of the form $(\epsilon,j)$ with $\epsilon \in \set{0,1}$.
  \end{itemize}
  Unraveling the definitions we see that we need to solve the associated lifting problem
  \[
  \begin{tikzcd}[ampersand replacement=\&]
  \mathcal{S}^{n+1} \arrow[d,"\iota"] \arrow[r,"\partial \alpha"] \& \bcat{D} \\
  \Delta^1\otimes \Delta^{n+1} \arrow[ur,dotted] \&
  \end{tikzcd}
  \]
  We will abuse notation and denote by $ \Delta^1\otimes \Delta^{n+1} $ the Gray product where we are additionally scaling the triangles $(0,j) \to (0,j+1) \to (1,j+1)$ whenever $j<n$ and the triangle $(n,0) \to (n+1,0) \to (1,n+1,1)$. We will carry this additional scaling to $\mathcal{S}^{n+1}$. Note that by construction $\partial \alpha$ sends those triangles to thin simplices in $\bcat{D}$. We produce a factorization 
  \[
  \func{\mathcal{S}^{n+1} \to[u] \mathcal{R}^{n+1} \to[v] \Delta^1\otimes\Delta^{n+1}}
  \]
  where $\mathcal{R}^{n+1}$ consists of those simplices of $\Delta^1\otimes \Delta^{n+1} $ that skip the vertex $(1,1)$. It is easy to see that $u$ is scaled anodyne and that $v$ fits into a pushout square
  \[
  \begin{tikzcd}[ampersand replacement=\&]
  \Lambda^{n+1}_0 \coprod\limits_{\Delta^{\set{0,1}}}\Delta^{0} \arrow[r] \arrow[d] \& \Delta^{n+1}\coprod\limits_{\Delta^{\set{0,1}}}\Delta^0 \arrow[d] \\
  \mathcal{R}^{n+1} \arrow[r,"v"] \& \Delta^1\otimes\Delta^{n+1}
  \end{tikzcd}
  \]
  since the triangle $\set{0,1,n}$ is thin by construction it follows that $v$ is also scaled anodyne. The result now follows.
\end{proof}

We now arrive at the main theorem of this section, which provides a computational criterion for cofinality. 

\begin{theorem}\label{thm:cofinality}
  Let $f:\bcat{C}^\dagger \to \bcat{D}^\dagger$ be a functor of marked $\infty$-bicategories. Then the following statements are equivalent
  \begin{enumerate}
    \item The functor $f$ is marked cofinal.
    \item For every $d \in \bcat{D}$ the functor $f$ induces an equivalence of $\infty$-categorical localizations $L_W(\bcat{C}_{d\upslash}^\dagger) \to L_W(\bcat{D}_{d\upslash}^\dagger)$.
    \item The following conditions hold:
    \begin{itemize}
      \item[i)] For every $d \in \bcat{D}$ there exists a morphism $g_d:d \to f(c)$ which is initial in $L_W(\bcat{C}^\dagger_{d \upslash})$ and  $L_W(\bcat{D}^\dagger_{d \upslash})$.
      \item[ii)] Every marked morphism $d \to f(c)$ defines an initial object in $L_W(\bcat{C}^{\dagger}_{d\upslash})$.
      \item[iii)] For any marked morphism $d \to b$ in $\bcat{D}$ the induced functor $L_W(\bcat{C}^\dagger_{b \upslash}) \to L_W(\bcat{C}^\dagger_{d \upslash})$ preserves initial objects.
    \end{itemize}
  \end{enumerate}
\end{theorem}
\begin{proof}
  By \autoref{prop:prebigtheorem} it will suffice to show that $2$ holds if and only if $3$ holds. Let us suppose that $2$ holds. Since by hypothesis the morphism $L_W(\bcat{C}^{\dagger}_{d\upslash}) \to L_W(\bcat{D}^{\dagger}_{d\upslash})$ is an equivalence of $\infty$-categories we can pick an object $d\to f(c)$ whose image in $L_W(\bcat{D}^{\dagger}_{d\upslash})$ is equivalent to $\on{id}_d$. Since equivalences preserve and detect initial objects we see that condition $i)$ is satisfied. To see that condition $ii)$ holds we just note that every marked morphism in $L_W(\bcat{D}^{\dagger}_{d\upslash})$ is equivalent to $\on{id}_d$. Using again that equivalences detect initial objects the claim follows. For the final condition we consider a commutative diagram
  \[
  \begin{tikzcd}[ampersand replacement=\&]
  L_W(\bcat{C}^{\dagger}_{b\upslash}) \arrow[r,"\simeq"] \arrow[d] \& L_W(\bcat{D}^{\dagger}_{b\upslash}) \arrow[d] \\
  L_W(\bcat{C}^{\dagger}_{d\upslash}) \arrow[r,"\simeq"] \& L_W(\bcat{D}^{\dagger}_{d\upslash})
  \end{tikzcd}
  \] 
  It is now clear that condition $iii)$ holds if the right-most vertical morphism preserves initial objects. We  observe that this map sends the identity on $b$ to an object represented by a marked morphism and thus preserves initial objects.
  
  Now let us suppose that the conditions in $3$ are satisfied. Using \autoref{prop:prebigtheorem} we see that it will suffice to show that the induced morphism of fibrant replacements $\mathcal{A}_f:\bcat{C}^\dagger_{\mathbb{D}\upslash} \to \bcat{D}^\dagger_{\bcat{D}\upslash}$ is an equivalence of 2-Cartesian fibrations. Notice that by assumption it follows that $\bcat{C}^{\dagger}_{\bcat{D}\upslash}$ satisfies the hypothesis of \autoref{prop:initial2version}. Let us denote by $\hat{\bcat{C}}^{\dagger}_{\bcat{D}\upslash}$ the full marked-biscaled simplicial set spanned by fibrewise initial objects and similarly for $\hat{\bcat{D}}^{\dagger}_{\bcat{D}\upslash}$. Observe that due to \autoref{prop:initial2version} we have a section
  \[
  \func{s_f:\bcat{D}_{\sharp}\to \hat{\bcat{C}}^{\dagger}_{\bcat{D}\upslash} \to \bcat{C}^\dagger_{\bcat{D}\upslash}}\enspace  \text{ where } \bcat{D}_\sharp=(\bcat{D},E_{\bcat{D}},T_{\bcat{D}}\subset \sharp).
  \]
  We can pick the section so that each $d$ gets sent to $g_d:d \to f(c)$ as in condition $i)$. We claim that $s_f$ sends marked edges in $\bcat{D}_{\sharp}^\dagger=(\bcat{D},E_{\bcat{D}^\dagger},T_{\bcat{D}^\dagger}\subset \sharp)$ to Cartesian edges in $\bcat{C}^\dagger_{\bcat{D}\upslash}$. Let $e:d \to b$ be a marked edge in $\bcat{D}_\sharp^\dagger$ and pick a Cartesian lift of $e$, $\hat{e}:\Delta^1 \to \bcat{C}^{\dagger}_{\bcat{D}\upslash}$ such that $\hat{e}(1)=g_b$. By condition $iii)$,  we have that $\hat{e}(0)$ is initial in the fibre over $d$. We consider the commutative diagram
  \[
  \begin{tikzcd}
  \Lambda^2_2 \arrow[r,"\sigma"] \arrow[d] & \bcat{C}^{\dagger}_{\bcat{D}\upslash} \arrow[d] \\
  \Delta^2 \arrow[r,"s_0(e)"] \arrow[ur,dotted,"\theta"] & \bcat{D}
  \end{tikzcd}
  \]
  with $\sigma(1\to 2)=\hat{e}$ and $\sigma(0\to 2)=s_f(e)$. The triangle $\theta$ is thin  by construction and the edge $0\to 1$ is an equivalence since it is a morphism between initial objects. It follows that $s_f(e)$ is Cartesian. We can now use  \autoref{lem:IffII} to produce a solution to the lifting problem
  \[
  \begin{tikzcd}
  \bcat{D}_\sharp^{\dagger} \arrow[d,"i_{\bcat{D}}"] \arrow[r,"s_f"] & \bcat{C}^\dagger_{\bcat{D}\upslash} \\
  \bcat{D}^\dagger_{\bcat{D}\upslash} \arrow[ur,dotted,swap,"\mathcal{I}_f"] &
  \end{tikzcd}
  \]
  We claim that $\mathcal{A}_f$ and $\mathcal{I}_f$ are mutually inverse. First we observe that  $\mathcal{A}_f \circ s_f$ is a section of $ \bcat{D}^\dagger_{\bcat{D}\upslash}$ that maps each object $d \in \bcat{D}$ to an initial object in the fibre. Using the fact that $\hat{\bcat{D}}^\dagger_{\bcat{D}\upslash} \to \bcat{D}$ is a trivial fibration we can construct a homotopy over $\bcat{D}$,  
  \[
  \func{H_{\bcat{D}}:\Delta^1 \times \bcat{D}_\sharp \to \bcat{D}^\dagger_{\bcat{D}\upslash}}    
  \]     
  between $i_{\bcat{D}}$ and $\mathcal{A}_f \circ s_f$. Observe that $i_{\bcat{D}}(d)=\on{id}_d$ so it  maps every object to an initial object in the fibre. By construction the components of $H_{{\bcat{D}}}$ are morphisms between initial objects and thus equivalences. Let $e:  d \to b$ be a marked morphism in ${\bcat{D}}_\sharp^\dagger$ then it follows that $H_{{\bcat{D}}}(0\to 1,e)$ is marked in $\bcat{D}^\dagger_{\bcat{D}\upslash}$. We can therefore upgrade the homotopy $H_{\bcat{D}}$  to a marked homotopy $H_{\bcat{D}}:(\Delta^1)^\sharp \times \bcat{D}_\sharp^\dagger \to \bcat{D}^\dagger_{\bcat{D}\upslash}$. To see that $\mathcal{A}_f \circ \mathcal{I}_f \simeq \on{id}$ it suffices to check that $\mathcal{A}_f \circ \mathcal{I}_f \circ i_{\bcat{D}} \isom i_{\bcat{D}}$, however we have
  \[
  \mathcal{A}_f \circ \mathcal{I}_f \circ i_{\bcat{D}} = \mathcal{A}_f \circ s_f \simeq i_{\bcat{D}}.
  \]
  Let us fix some notation $i_f: \bcat{C}_{\sharp}^\dagger \to \bcat{C}^\dagger_{\bcat{D}\upslash}$ and $i_{\bcat{C}}: \bcat{C}_{\sharp}^{\dagger} \to \bcat{C}^\dagger_{\bcat{C}\upslash}$. In order to show that $\mathcal{I}_f \circ \mathcal{A}_f \isom \on{id}$ we will show that $\mathcal{I}_f \circ \mathcal{A}_f \circ i_f \isom i_f$. Consider the following pullback square
  \[
  \begin{tikzcd}
  f^{*}\left(\bcat{C}^\dagger_{\bcat{D}\upslash}\right) \arrow[d] \arrow[r,"\phi"] & \bcat{C}^\dagger_{\bcat{D}\upslash} \arrow[d] \\
  \bcat{C}  \arrow[r,"f"] & \bcat{D}
  \end{tikzcd}
  \]
  and note that $i_f= \phi \circ \mathcal{B}_f\circ  i_{\bcat{C}}$ where $\mathcal{B}_f:\bcat{C}^\dagger_{\bcat{C}\upslash} \to f^{*}\left(\bcat{C}^\dagger_{\bcat{D}\upslash}\right)$ is the morphism induced by the universal property. We now observe that $\mathcal{I}_f \circ \mathcal{A}_f  \circ \phi=\phi \circ f^*(\mathcal{I}_f) \circ f^*(\mathcal{A}_f)$. A similar argument as before shows that 
  \[
  f^*(\mathcal{I}_f) \circ f^*(\mathcal{A}_f) \circ \mathcal{B}_f \circ i_{\bcat{C}} \simeq \mathcal{B}_f \circ i_{\bcat{C}}.
  \]
  This is due to the fact that both sides of the equation describe sections of $f^{*}\left(\bcat{C}^\dagger_{\bcat{D}\upslash}\right)$ with values in initial objects. Note that $\mathcal{B}_f \circ i_{\bcat{C}} (c)$ is initial as a consequence of condition $ii)$. We conclude the proof by finally noting 
  \[
  \mathcal{I}_f \circ \mathcal{A}_f \circ i_f=\mathcal{I}_f \circ \mathcal{A}_f \circ \phi \circ \mathcal{B}_f\circ  i_{\bcat{C}} =\phi \circ f^*(\mathcal{I}_f) \circ f^*(\mathcal{A}_f) \circ  \mathcal{B}_f \circ i_{\bcat{C}} \simeq \phi \circ \mathcal{B}_f \circ i_{\bcat{C}}=i_f
  \]
  We have shown that $\mathcal{I}_f \circ \mathcal{A}_f  \simeq \on{id}$ and the theorem now follows.
\end{proof}

We derive as an immediate corollary a $\infty$-bicategorical upgrade of Quillen's Theorem A.

\begin{corollary}\label{cor:Adagger}
  Let $f:\bcat{C}^\dagger \to \bcat{D}^\dagger$ be a functor of marked $\infty$-bicategories and suppose that for every $d \in \bcat{D}$ the induced functor 
  \[
    \func{L_W(\bcat{C}^\dagger_{d\upslash}) \to[\simeq] L_W(\bcat{D}^\dagger_{d\upslash})}
  \]
  is an equivalence of $\infty$-categories. Then the functor $f$ induces an equivalence upon passage to $\infty$-categorical localizations 
  \[
    \func{L_W(\bcat{C}^\dagger) \to[\simeq] L_W(\bcat{D}^\dagger)}.
  \]
\end{corollary}
\begin{proof}
  By \autoref{thm:cofinality} it follows that $f$ is marked cofinal. Since the pushforward functor 
  \[
    \func{t_*:\left(\on{Set}_\Delta^{\mathbf{mb}}\right)_{/\bcat{D}} \to \on{Set}_\Delta^{\mathbf{mb}}}
  \]
  is left Quillen and every object is cofibrant, $t_\ast$ preserves weak equivalences. The image of the map $f:(\bcat{C},E_{\bcat{C}},T_{\bcat{C}}\subset \sharp) \to (\bcat{D},E_{\bcat{C}},T_{\bcat{D}}\subset \sharp)$ under $t_*$ is equivalent to the morphism $t_*(f):(\bcat{C},E_{\bcat{C}}, \sharp) \to (\bcat{D},E_{\bcat{D}},\sharp)$. The claim follows after taking a fibrant replacement of the map $t_*(f)$.
\end{proof}

We finish this section by studying the case where $f: \CC^\dagger \to \DD^\dagger$ is a functor between strict 2-categories equipped with a marking. For the rest of the section we will denote $\bcat{C}^\dagger=\Nsc(\CC)^\dagger$ (resp. $\bcat{D}^\dagger:=\Nsc(\DD)^\dagger$) where the marking comes from the marking in $\CC^\dagger$ (resp. $\DD^\dagger$). Our goal is to relate (nerves of)  the comma 2-categories of \autoref{defn:upslice} with the fibres of the free 2-Cartesian fibration thus simplifying the conditions of \autoref{thm:cofinality}.

\begin{definition}
  Let $f:\CC \to \DD$ be a functor of strict 2-categories. We define a new 2-category $\Fr(\CC)$ as follows:
  \begin{itemize}
    \item Objects are given by morphisms $u:d_0 \to f(c_0)$ where $d_0 \in \DD$ and $c_0 \in \CC$.
    \item A morphism $\varphi_0:u \to v$ from $u: d_0 \to f(c_0)$  to $v:d_1 \to f(c_1)$ is given by a pair of morphisms $a_0: d_0 \to d_1$ and $\alpha_0:c_0 \to c_1$ and a 2-morphism $\theta_{\varphi_0}:f(\alpha)\circ u \xRightarrow{}v \circ a$.
    \item A 2-morphism  $\epsilon:\varphi_0 \to \varphi_1$ is given by a pair of 2-morphisms $\psi:a_0 \xRightarrow{} a_1$ and $\zeta: \alpha_0 \xRightarrow{} \alpha_1$ such that the followign diagram commutes
    \[
      \begin{tikzcd}
        f(\alpha_0)\circ u \arrow[r,nat,"f(\zeta)* u"] \arrow[d,nat,"\theta_{\varphi_0}"] & f(\alpha_1) \circ u \arrow[d,nat,"\theta_{\varphi_1}"]\\
          v \circ a_0 \arrow[r,nat,"v* \psi"] & v \circ a_1
      \end{tikzcd}
    \]
  \end{itemize}
There is an obvious 2-functor $\Fr(\CC) \to \DD$ which is easily verified to be a 2-Cartesian fibration. In particular one observes the following:
\begin{itemize}
  \item A morphism in $\Fr(\CC)$ is Cartesian if the associated morphism $\alpha: c_0 \to c_1$ is an equivalence in $\CC$ and the 2-morphism $\varphi_0$ is invertible.
  \item A 2-morphism in $\Fr(\CC)$ is coCartesian if the associated 2-morphism $\zeta: \alpha_0 \xRightarrow{} \alpha_1$ is invertible.
\end{itemize}
One immediately sees that the fibres of $\Fr(\CC)$ are precisely the categories $\CC_{d\upslash}$ of \autoref{defn:upslice}.
\end{definition}

\begin{remark}
  As a direct consequence of \cite[Theorem 4.29]{AGS_CartI} we see that the induced morphism $\Nsc(\Fr(\CC))\to \bcat{D}$ is a 2-Cartesian fibration. We further observe that there is an strict 2-functor $\CC \to \Fr(\CC)$. We will see at the end of the section that $\Fr(\CC)$ is another model for the free 2-Cartesian fibration on the functor $f$.
\end{remark}

\begin{remark}
  Suppose we are given a morphism of marked strict 2-categories $f: \CC^\dagger \to \DD^\dagger$. Then we can construct a marked 2-category $\Fr(\CC)^\dagger$ by declaring an edge in $\Fr(\CC)$ to be marked if and only if it is Cartesian or the associated morphism $\alpha:c_0 \to c_1$ is marked in $\CC^\dagger$ and the 2-morphism $\varphi_0$ is invertible. We denote by $\Nsc(\Fr(\CC))^\dagger$ the associated $\bS$ simplicial set.
\end{remark}

Before we continue, we must provide a good characterization of the simplices of $\Nsc(\Fr(\CC))$. As it turns out, we can view $\Nsc(\Fr(\CC))$ as a simplicial subset of $\FF(\bcat{C})$, and we will use this to provide an alternate characterization of the simplices of the former. To this end we fix some terminology. Let us call a 2-simplex of $\Delta^1\widehat{\otimes} \Delta^n$ \emph{contrary} if it is scaled. 

An $n$-simplex $\sigma$ of $\FF(\bcat{C})$  consists of a commutative diagram 
\[
\begin{tikzcd}
\Delta^1\widehat{\otimes}\Delta^n \arrow[r,"\phi^\sigma"] & \bcat{D}\\
\{1\}\times \Delta^n \arrow[r,"\rho^\sigma"]\arrow[u] & \bcat{C}\arrow[u,"f"'] 
\end{tikzcd}
\]
We will call such a simplex \emph{tame} if it sends contrary $2$-simplices to \emph{identities}.\footnote{Notice that this definition is only sensible because $\bcat{D}:=\Nsc(\DD)$. Otherwise, there is no good notion of identity 2-simplices which are neither left nor right degenerate.} By construction, the tame simplices form a simplicial subset of $\FF(\bcat{C})$, which we will denote by $\on{Tame}(\bcat{C},\bcat{D})$. When this is equipped with the marking and biscaling induced by $\FF(\bcat{C})$, we denote it by $\on{Tame}(\bcat{C},\bcat{D})^\dagger$.

\begin{lemma}
  There is an isomorphism 
  \[
  \on{Tame}(\bcat{C},\bcat{D})^\dagger\cong \Nsc(\Fr(\CC))^\dagger
  \]
  of marked-biscaled simplicial sets.
\end{lemma}

\begin{proof}
  We will prove that the underlying simplicial set $\on{Tame}(\bcat{C},\bcat{D})$ is 3-coskeletal, reducing the proof to a straightforward check on 3-truncations. 
  
  Consider a morphism $\partial \Delta^n \to \on{Tame}(\bcat{C},\bcat{D})$ where $n>3$. This corresponds to a diagram 
  \[
  \begin{tikzcd}
  \Delta^1\times \partial \Delta^n \arrow[r,"\phi"] & \bcat{D}\\
  \{1\}\times \partial \Delta^n \arrow[r,"\rho"'] \arrow[u] & \bcat{C}\arrow[u,"f"']
  \end{tikzcd}
  \]
  We now note that $\bcat{C}$ and $\bcat{D}$ are, themselves 3-coskeletal, and thus, in particular, they admit unique horn fillers for all horns of dimension 5 or higher. Using, e.g., the filtration of \cite[Prop. 2.1.2.6]{HTT}, we see that $\phi$ has a unique extension to a map $\Delta^1\times \Delta^n\to \bcat{D}$. Since $\rho$ clearly has a unique extension to a map $\Delta^n\to \bcat{C}$, we can obtain an extension 
  \begin{equation}\label{eqn:bndryinslice}
  \begin{tikzcd}
  \Delta^1\times  \Delta^n \arrow[r,"\widetilde{\phi}"] & \bcat{D}\\
  \{1\}\times \Delta^n \arrow[r,"\widetilde{\rho}"'] \arrow[u] & \bcat{C}\arrow[u,"f"']
  \end{tikzcd}\tag{$\ast$}
  \end{equation}
  of the diagram above. 
  
  Moreover, since $n>3$, every 2-simplex of $\Delta^1\times \Delta^n$ is contained in $\Delta^1\times \partial \Delta^n$. Consequently, the fact that $\phi$ arises from a map $\partial\Delta^n\to \on{Tame}(\bcat{C},\bcat{D})$ implies that the diagram (\ref{eqn:bndryinslice})  defines an $n$-simplex in $\on{Tame}(\bcat{C},\bcat{D})$. 
  
  The remaining low-dimensional checks are left to the reader.
\end{proof}

\begin{remark}
  Note that the argument above can in fact be repurposed to show that $\FF(\bcat{C})$ is itself 3-coskeletal in our present setting. However, in spite of their equivalence,  $\FF(\bcat{C})$ will not be isomorphic to $\Nsc(\Fr(\CC))$, as the former has significantly more 1- and 2-simplices.
\end{remark}

\begin{remark}
  By construction, the canonical morphism $\bcat{C}\to \FF(\bcat{C})$ factors through $\on{Tame}(\bcat{C},\bcat{D})$. 
\end{remark}

\begin{lemma}
  Let $\sigma:\Delta^n\to \on{Tame}(\bcat{C},\bcat{D})$ be an $n$-simplex. Then each extension $E_j^\ast(\sigma):\Delta^{n+1}\to \FF(\bcat{C})$ factors through $\on{Tame}(\bcat{C},\bcat{D})$.  
\end{lemma}

\begin{proof}
  This follows immediately from unraveling the definitions.
\end{proof}

\begin{definition}
  We denote by $\on{Tame}(\bcat{C},\bcat{D})^\natural$ the marking and biscaling induced by $\FF(\bcat{C})^\natural$. Similarly, we denote by $\on{Tame}(\bcat{C},\bcat{D})^\dagger$ the marking and biscaling induced by $\FF(\bcat{C})^\dagger$.
\end{definition}

\begin{proposition}\label{prop:tameEquivC}
  The morphism $\bcat{C}^\natural \to \on{Tame}(\bcat{C},\bcat{D})^\natural$ is \textbf{MB}-anodyne over $\bcat{D}$. 
\end{proposition}

\begin{proof}
  This is identical to the proof of \autoref{thm:thebigfreeboi} once we redefine $Z_n$ to consist of $n$-simplices of $\on{Tame}(\bcat{C},\bcat{D})$, together with $j$-extensions of these simplices.
\end{proof}

\begin{theorem}\label{thm:comparisonofslices}
  Let $f:\CC^\dagger \to \DD^\dagger$ be a functor between marked strict 2-categories, and let $f:\bcat{C}^\dagger\to \bcat{D}^\dagger$ denote the induced morphism of \textbf{MS} simplicial sets. Then the following hold:
  \begin{itemize}[noitemsep]
    \item There exists a commutative diagram over $\bcat{D}$
    \[
      \begin{tikzcd}
        \bcat{C}^\dagger \arrow[r] \arrow[d] & \Nsc(\Fr(\CC)) \\
        \mathbb{F}(\bcat{C})^\dagger \arrow[ur,"\Xi",swap] &
      \end{tikzcd}
    \]
    such that each morphism in the diagram is a 2-Cartesian equivalence.
    \item For every $d \in \DD$ the map $\Xi$ induces an equivalence of \textbf{MS} simplicial sets $\bcat{C}_{d\upslash}^\dagger \xrightarrow{\simeq} \Nsc(\CC_{d\upslash}^{\dagger})$.
  \end{itemize}
\end{theorem}

\begin{proof}
  First let us assume that the marking on both $\CC^\dagger=\CC^\natural$ and $\DD^\dagger=\DD^\dagger$ only consists of equivalences so that the marking on both $\mathbb{F}(\bcat{C}) $ and $\Nsc(\Fr(\CC)^\natural)$ is precisely given by Cartesian edges. Recall the filtration defined in \autoref{thm:thebigfreeboi}. First we will define a morphism
   \[
      \begin{tikzcd}
        \bcat{C}^\natural \arrow[r] \arrow[d] & \Nsc(\Fr(\CC)^\natural) \\
        Z_1 \arrow[ur,"\Xi_1",swap] &
      \end{tikzcd}
    \]
    Since $Z_1 \to \FF(\bcat{C})$ is \bS-anodyne and $\Nsc(\Fr(\CC))$ is a 2-Cartesian fibration, we can pick an extension of $\Xi_1$ to the desired $\Xi$. Observe that we can map the objects of $Z_1$ isomorphically to those of $\Nsc(\Fr(\CC))$. Given  $e:\Delta^1 \to Z_1$ we see that this data precisely amounts to morphisms $u_i:d_i \to f(c_i)$ for $i=0,1$, $a: d_0 \to d_1$, $\alpha:c_0 \to c_1$ and $g:d_0 \to f(c_1)$ together with a pair of 2-morphisms $\epsilon:g \xRightarrow{\simeq}f(\alpha) \circ u_0$ and $\theta: g \xRightarrow{} u_1 \circ a$ such that $\epsilon$ is invertible. We can then map $e$ to an edge $\Xi_1(e)$ defined by the same 1-morphisms but with associated 2-morphism $\theta \circ \epsilon^{-1}$. One perfoms a similar construction for mapping the non-degenerate 2-simplices contained in $Z_1$ thus giving a definition for $\Xi_1$.

    We can now observe that in the case that $\CC^\dagger$ comes equipped with a general marking (containing the equivalences) we have a homotopy pushout $(\on{Set}^{\mathbf{mb}}_\Delta)_{/\bcat{D}}$
    \[
      \begin{tikzcd}[ampersand replacement=\&]
        \mathbb{F}(\bcat{C})^\natural \arrow[r] \arrow[d] \& \Nsc(\Fr(\CC)^\natural) \arrow[d] \\
          \mathbb{F}(\bcat{C})^\dagger \arrow[r] \& \Nsc(\Fr(\CC)^\dagger) 
      \end{tikzcd}
    \]
    This shows it will suffice to prove the case where only the equivalences are marked. This follows from 2-out-of-3 after noting that $\bcat{C}^\natural \to \Nsc(\Fr(\CC)^\natural)$ is an equivalence. This follows immediately from \autoref{prop:tameEquivC}.
\end{proof}

\begin{remark}
  The significance of this result is twofold. Most importantly, it shows that, when considering diagrams indexed over strict 2-categories, the criteria for marked cofinality can be expressed in terms of the strict slice 2-categories. Consequently, the criteria for cofinality become much easier to explicitly check in this case.
  
  Of lesser significance, but still of interest, is the second consequence. Since we can identify $\Nsc(\CC_{d\upslash})$ and $\bcat{C}_{d\upslash}$, the criteria of \autoref{thm:cofinality} precisely agree with those of \cite[Thm 4.0.1]{AGSQuillen}. \autoref{thm:cofinality} thus generalizes \cite[Thm 4.0.1]{AGSQuillen}, as expected.
\end{remark}

  \end{document}